\documentclass[a4paper,11pt,reqno]{amsart}
\usepackage{amsmath}
\usepackage{amsfonts}
\usepackage{amsthm}
\usepackage{amssymb}
\usepackage{tikz}
\usepackage[all]{xy}
\usepackage{hyperref}
\hypersetup{ 
    colorlinks=false,
    pdfborder={0 0 0},
    unicode=true}
\newcommand{\nref}[1]{\hyperref[#1]{\ref*{#1}}}

\theoremstyle{plain}
\newtheorem{theorem}{Theorem}[section]

\newtheorem{corollary}[theorem]{Corollary}

\newtheorem{definition}[theorem]{Definition}

\newtheorem{example}[theorem]{Example}

\newtheorem{lemma}[theorem]{Lemma}

\newtheorem{proposition}[theorem]{Proposition}
\newtheorem{remark}[theorem]{Remark}

\newcommand{\hamburger}[4] 
{
  \thispagestyle{empty}
  \vspace*{-2cm}
  \begin{flushright}
    ZMP-HH 14/ #2 \\
    Hamburger Beitr\"age zur Mathematik Nr. #3 \\
    #4 \\
  \end{flushright}
  \vspace{0.5cm}
  \begin{center}
    \Large \bf
    #1
  \end{center}
  \vspace{0.5cm}
  \begin{center}        
    Alexander Barvels, Simon Lentner, Christoph Schweigert \\
    Fachbereich Mathematik, Universit\"at Hamburg \\
    Bereich Algebra und Zahlentheorie \\
    Bundesstra\ss e 55, D-20146 Hamburg \\
  \end{center}
  \vspace{0.5cm}
}

\newcommand{\id}{{\rm id}}
\newcommand{\Id}{{\rm Id}}
\newcommand{\eps}{\varepsilon}
\newcommand{\ad}{{\rm ad}}

\newcommand{\cA}{\mathcal{A}}
\newcommand{\cB}{\mathcal{B}}
\newcommand{\cC}{\mathcal{C}}
\newcommand{\cD}{\mathcal{D}}

\newcommand{\cZ}{\mathcal{Z}}

\newcommand{\fg}{\mathfrak{g}}

\renewcommand{\C}{\mathbb{C}}

\newcommand{\Z}{\mathbb{Z}}
\newcommand{\unit}{\mathbf{1}}

\newcommand{\sD}{\mathsf{D}}
\newcommand{\sS}{\mathsf{S}}
\newcommand{\sT}{\mathsf{T}}

\newcommand{\cop}{{{\rm cop}}}
\newcommand{\op}{{{\rm op}}}

\newcommand{\ov}[1]{\overline{#1}}

\newcommand{\us}{\underline{\; \;}}

\newcommand{\YD}[1]{{}_{#1}^{#1}\mathcal{YD}}
\newcommand{\rYD}[1]{\mathcal{YD}_{#1}^{#1}}
\newcommand{\lrYD}[2]{{}_{#1}\mathcal{YD}^{#2}}
\newcommand{\rlYD}[2]{{}^{#2}\mathcal{YD}_{#1}}
\newcommand{\lMod}[2]{#1\text{-}\mathrm{mod}_{#2}}
\newcommand{\rMod}[2]{\mathrm{mod}_{#2}\text{-}#1}
\newcommand{\lComod}[2]{#1\text{-}\mathrm{comod}_{#2}}
\newcommand{\rComod}[2]{\mathrm{comod}_{#2}\text{-}#1}
\newcommand{\Bimod}[2]{#1\text{-}{\rm bimod}_{#2}}

\newcommand{\grau}{gray!60}
\newenvironment{grform}{\begin{tikzpicture}[intext]}{\end{tikzpicture}}
\tikzset{
norm/.style = {ultra thick, color = #1},
norm/.default = black,
intext/.style = {baseline = (current bounding box.center)},
}

\newcommand{\vLine}[5]{
  \draw[ultra thick, color = #5, rounded corners] (#1 , #2) -- (#1 , #2 * 0.7 +
#4 * 0.3 ) -- ( #3 , #2 * 0.3 + #4 * 0.7 ) -- (#3 ,#4);
}
\newcommand{\vLineO}[5]{
  \draw[line width = 6pt , color = white, rounded corners] (#1 , #2) -- (#1 , #2
* 0.7 + #4 * 0.3 ) -- ( #3 , #2 * 0.3 + #4 * 0.7 ) -- (#3 ,#4);
  \draw[ultra thick, color = #5, rounded corners] (#1 , #2) -- (#1 , #2 * 0.7 +
#4 * 0.3 ) -- ( #3 , #2 * 0.3 + #4 * 0.7 ) -- (#3 ,#4);
}

\newcommand{\dMult}[5]{
\draw[ultra thick, color = #5] (#1 , #2) -- (#1 , #2 + #4 * 0.2) .. controls (#1
, #2 + #4 * 0.5) .. (#1 + #3 *0.25 , #2 + #4 * 0.5) 
-- (#1 + #3 *0.5 , #2 + #4 * 0.5) -- (#1 + #3 *0.75 , #2 + #4 * 0.5) .. controls
(#1 + #3 , #2 + #4 * 0.5) .. (#1 + #3 , #2 + #4 * 0.2) -- (#1 + #3 , #2);
\draw[ultra thick, color = #5] (#1 + #3 *0.5 , #2 + #4 * 0.5) -- (#1 + #3 *0.5 ,
#2 + #4);
}

\newcommand{\dMultO}[5]{
\draw[line width = 6pt , color = white] (#1 , #2) -- (#1 , #2 + #4 * 0.2) ..
controls (#1 , #2 + #4 * 0.5) .. (#1 + #3 *0.25 , #2 + #4 * 0.5) 
-- (#1 + #3 *0.5 , #2 + #4 * 0.5) -- (#1 + #3 *0.75 , #2 + #4 * 0.5) .. controls
(#1 + #3 , #2 + #4 * 0.5) .. (#1 + #3 , #2 + #4 * 0.2) -- (#1 + #3 , #2);
\draw[ultra thick, color = #5] (#1 , #2) -- (#1 , #2 + #4 * 0.2) .. controls (#1
, #2 + #4 * 0.5) .. (#1 + #3 *0.25 , #2 + #4 * 0.5) 
-- (#1 + #3 *0.5 , #2 + #4 * 0.5) -- (#1 + #3 *0.75 , #2 + #4 * 0.5) .. controls
(#1 + #3 , #2 + #4 * 0.5) .. (#1 + #3 , #2 + #4 * 0.2) -- (#1 + #3 , #2);
\draw[ultra thick, color = #5] (#1 + #3 *0.5 , #2 + #4 * 0.5) -- (#1 + #3 *0.5 ,
#2 + #4);
}

\newcommand{\dUnit}[3]{
\draw[ultra thick, color = #3] (#1 , #2) -- (#1 , #2 - 0.5);
\filldraw[ultra thick, color = #3, fill = white] (#1 , #2 - 0.5) circle [radius
= 6pt];
}

\newcommand{\dCounit}[3]{
\draw[ultra thick, color = #3] (#1 , #2) -- (#1 , #2 + 0.5);
\filldraw[ultra thick, color = #3, fill = white] (#1 , #2 + 0.5) circle [radius
= 6pt];
}

\newcommand{\dAntipode}[3]{
\draw[ultra thick, color = #3] (#1 , #2) -- (#1 , #2 + 1);
\filldraw[ultra thick, color = #3, fill = white] (#1 , #2 + 0.5) circle [radius
= 8pt];
}

\newcommand{\dSkewantipode}[3]{
\draw[ultra thick, color = #3] (#1 , #2) -- (#1 , #2 + 1);
\filldraw[ultra thick, color = #3, fill = gray] (#1 , #2 + 0.5) circle [radius =
8pt];
}

\newcommand{\dAction}[6]{
\draw[ultra thick, color = #5] (#1 , #2) -- (#1 , #2 + #4 * 0.2) .. controls (#1
, #2 + #4 * 0.5) .. (#1 + #3 * 0.8 , #2 + #4 * 0.5) -- (#1 + #3 , #2 + #4 *
0.5);
\draw[ultra thick, color = #6] (#1 + #3 , #2) -- (#1 + #3 , #2 + #4);
}

\newcommand{\dPairing}[6]{
\draw[ultra thick, color = #5] 
(#1 , #2) -- 
(#1 , #2 + #4 - 1) .. controls (#1 , #2 + #4 - 0.5) .. (#1 + #3 * 0.48 , #2 + #4
- 0.5) -- 
(#1 + #3 * 0.52 , #2 + #4 - 0.5) .. controls (#1 + #3 , #2 + #4 - 0.5) .. (#1 +
#3 , #2 + #4 - 1) -- 
(#1 + #3 , #2);
\draw (#1 + #3 * 0.5 , #2 + #4 - 0.5) node [fill = white] {#6};
}

\newcommand{\dCopairing}[6]{
\draw[ultra thick, color = #5] 
(#1 , #2) -- 
(#1 , #2 - #4 + 1) .. controls (#1 , #2 - #4 + 0.5) .. (#1 + #3 * 0.48 , #2 - #4
+ 0.5) -- 
(#1 + #3 * 0.52 , #2 - #4 + 0.5) .. controls (#1 + #3 , #2 - #4 + 0.5) .. (#1 +
#3 , #2 - #4 + 1) -- 
(#1 + #3 , #2);
\draw (#1 + #3 * 0.5 , #2 - #4 + 0.5) node [fill = white]
{#6};
}

\begin{document}

\hamburger{Partially dualized Hopf algebras have equivalent Yetter-Drinfel'd
modules}{05}{505}{February 2014}
\thispagestyle{empty}
\enlargethispage{1cm}

\begin{abstract}
Given a Hopf algebra $H$ and a projection $H\to A$ to a Hopf 
subalgebra, 
we construct a Hopf algebra $r(H)$, called the partial dualization
of $H$, 
with a projection to the Hopf algebra dual to $A$. This 
construction provides powerful techniques in the general 
setting of braided monoidal categories. The construction
comprises in particular the reflections of generalized quantum groups
\cite{HS13}. We prove a braided equivalence between the
Yetter-Drinfel'd modules over a Hopf algebra and its partial
dualization.
\end{abstract}

\makeatletter
\@setabstract
\makeatother

\tableofcontents
\newpage

\section{Introduction and Summary}

\subsection{Introduction}

One of the fundamental observations about a finite-dimensional
Hopf algebra $H$ over a field $\Bbbk$ is the fact that
 the dual vector space $H^*$ has the structure of a
Hopf algebra as well. The Hopf algebras $H$ and $H^*$ are, typically,
rather different.

In this article we consider a partial dualization of
a Hopf algebra $H$ in which the following data $\cA$ enter:
a projection    $\pi:H\rightarrow A$ to a Hopf subalgebra $A$, 
and a Hopf algebra $B$  dual to the subalgebra $A$. 
The duality is expressed  in terms of a non-degenerate Hopf 
pairing $\omega : A \otimes B \to \unit$.  these data, we
construct another Hopf algebra $r_{\cA}(H)$ which has $B$ as a Hopf subalgebra and comes with a projection to $B$. 
We call $r_\cA(H)$ the \emph{partially dualized} Hopf algebra. 
As we will see in Subsection \ref{subsec:examples}, 
such partially dualized Hopf algebras
appear in rather different contexts, including the  
classification of certain pointed Hopf algebras in terms
of their Borel parts, the Nichols algebras
\cite{AS02,Heck09,AHS10}.

The guiding example of Nichols algebras suggests the
following setting for our investigations:
for any braided category $\cC$, there is a natural notion 
of a Hopf algebra in $\cC$. In the present paper, $H,A$ and $B$ will
be  Hopf algebras in a braided category $\cC$. Working with 
Hopf algebras in braided categories 
significantly simplifies the motivating construction for 
generalized quantum groups  \cite{HS13}, since it avoids
explicit calculations with smash products. 

Our second aim is to exhibit a representation-theoretic relation between
the Hopf algebras $H$ and $r_\cA(H)$. Typically
$H$ and $r_\cA(H)$ are not isomorphic, nor even Morita-equivalent. In the
present paper, we show the following, more subtle relation:
    The categories of Yetter-Drinfel'd modules over $H$ and $r_\cA(H)$ are
    equivalent as braided categories. This insight is new, even in the case of
    generalized quantum groups. 

    The equivalence of categories of Yetter-Drinfel'd modules implies
    a relation between 
 the Hopf algebra $H$ and and its partial dualization $r_\cA(H)$ which we discuss
    in the case of a Hopf algebra over a field: the category of Yetter-Drinfel'd
    modules is the Drinfel'd center of the category of modules. It is well-known
    that semisimple algebras with isomorphic centers are Morita equivalent.
    Replace the notion of an algebra by the one of a monoidal category and,
    similarly, the notion of a module over a algebra by the notion of a module
    category over a monoidal category. Then, it is known \cite[Thm 3.1]{ENO11} that
    semisimple tensor categories with braided-equivalent Drinfel'd centers 
    have equivalent bicategories of module categories. This relation has been
    termed weak monoidal Morita equivalence \cite{Mue03}. 
It is therefore tempting to speculate that the bicategories
of module categories over the monoidal categories $\lMod{H}{}$ 
and $\lMod{r_\cA(H)}{}$ are closely related, if not equivalent.

\subsection{Examples}\label{subsec:examples}
We discuss several examples of partial dualizations and
relate them to known results in the literature; all examples
will be discussed in more detail in section  \nref{sec:Examples}.
The two
extremal cases of dualizations are trivial:

\begin{itemize}
\item Taking $\pi:A\stackrel\sim\to H$ yields a complete dualization: $r_\cA(H)=B$.
\item 
For the projection $\pi:H\rightarrow \unit_\cC=:A$ to the monoidal
unit $\unit$ of $\cC$,
we get $r_\cA(H)=H$.
\end{itemize}

\medskip

{\it Group algebras} of a finite groups already
provide examples of non-trivial partial dualizations.
Consider the complex group algebra $\C[G]$ of a finite
group $G$ which we assume to be  a semi-direct product $G=N\rtimes Q$.
As a consequence, $\C[G]$ is a Radford biproduct: 
$\C[G] = \C[N] \rtimes \C[Q]$ with  a trivial coaction and 
a non-trivial action of $\C[Q]$ on $\C[N]$. The (cocommutative) 
Hopf algebra $\C[Q]$ is dual to the (commutative) Hopf algebra 
$\C^Q$ of complex functions on $Q$.
The partial dualization with respect to the Hopf subalgebra 
$\C[Q]$ yields a Hopf algebra $\C[N] \rtimes \C^Q$ with a 
trivial action and a non-trivial coaction of $\C^Q$ on $\C[N]$.
The partially dualized Hopf algebra is neither a group algebra nor 
a dual group algebra. 

The monoidal category of modules over the partially dualized
Hopf algebra  $r_\cA(\C[G])=\C[N] \rtimes \C^Q$ turns out to be 
monoidally equivalent to the category of bimodules over
an algebra in the category $\mathrm{vect}_G$ of $G$-graded
vector spaces. Our general
result thus implies that the Drinfel'd center of the category $\lMod{\C[G]}{}$ and 
the Drinfel'd center of the category of bimodules 
are
braided equivalent.
This equivalence is a special case of \cite[Theorem 3.3]{Schau01}.

\medskip
 
The {\it Taft algebra} $T_\zeta$, with $\zeta$ a primitive $d$-th root 
of unity, is the Hopf algebra generated by a group-like element 
$g$ of order $d$, and a skew-primitive element $x$ with coproduct
$\Delta(x)=g\otimes x+x\otimes 1$. It has a projection 
$\pi$ to the Hopf subalgebra $A\cong\C[\mathbb{Z}_d]$
generated by $g$. The partial dualization $r_\cA(H)$ is 
isomorphic to the Taft algebra itself; the isomorphism 
depends on a choice of a Hopf pairing $\omega:A\otimes A\to\C$
and thus on a primitive $d$-th root of unity. An
example with non-trivial partial dualization is provided by a  
central extension
$\hat{T}_{\zeta,q}$ of the Taft algebra $T_\zeta$ by group-like elements.
The partial dualization $r_\cA(\hat{T}_{\zeta,q})=:\check{T}_{\zeta,q}$ 
then does not possess such central group-like elements;
instead,  the coproduct of the skew-primitive element of 
$\check{T}_{\zeta,q}$ is modified, leading to 
additional central characters for the partially dualized Hopf algebra $\check{T}_{\zeta,q}$.

\medskip

The {\it reflection of generalized quantum groups} as 
introduced in \cite{AHS10,HS13} was the original motivation of our construction.
In this case, the braided category $\cC$ is a category of Yetter-Drinfel'd
modules over a complex Hopf algebra $h$, i.e.\
$\cC=\YD{h}(\mathrm{vect}_\C)$. Usually, the complex Hopf algebra
$h$ is the complex group algebra of a finite group, $h=\C[G]$.
Next, fix a semisimple object 
$$M=M_1\oplus M_2 \oplus\cdots M_n$$
in $\cC$ and consider the Nichols algebra $\cB(M)$. It
is a Hopf algebra in the braided category $\cC$ and
plays the role of a quantum Borel part 
of a pointed Hopf algebra.

For each simple subobject $M_i$ in the direct decomposition of $M$, 
the Nichols algebra $A_i:=\cB(M_i)$ is a subalgebra of
$\cB(M)$; moreover, there is a natural projection of Hopf algebras 
$\pi_i: \cB(M)\to \cB(M_i)$. The Nichols algebra 
$\cB(M_i^*)$ for the object $M_i^*$ in $\cC$ dual to $M_i$ comes with
a non-degenerate Hopf pairing $\omega_i:\cB(M_i)\otimes\cB(M_i^*)\to\unit$.
We can thus perform a partial dualization. 
If the Nichols algebra $\cB(M)$ plays the role of a quantum Borel part,
the partially dualized Hopf algebra of $\cB(M)$ is
isomorphic to a quantum Borel part of in
$u_q(\fg)$ after a reflection on a simple root.

Examples are known that do not correspond to semisimple Lie algebras 
\cite{Heck09} and that can have a non-abelian Cartan subalgebra
\cite{AHS10,HS10a}. In these cases, a Borel subalgebra $H$ 
and its reflection $r_\cA(H)$ are not necessarily isomorphic.
We exhibit an explicit example in Section \ref{sec:ExampleNicholsAlgebra}

\subsection{Structure of the article and summary of results}

Section \nref{sec:Preliminaries} contains an overview of the 
theory of Hopf algebras in braided categories.
Some readers may prefer to skip Sections \ref{sec:Preliminaries} and
\ref{sec:YDM}, assuming that the braided category is the one of complex
vector spaces
$\mathcal{C}=\mathrm{vect}_\mathbb{C}$, and thereby restricting 
themselves to the case when $H$ is a finite-dimensional 
complex Hopf algebra. (This setting is not general enough to
cover the example of pointed Hopf algebras, though.)

In Section \nref{sec:YDM}, we review the notion of Yetter-Drinfel'd
modules over a Hopf algebra $A$ in a braided category 
$\cC$ as defined in \cite{Besp95}. The category of Yetter-Drinfel'd modules
is a braided category $\YD{A}(\mathcal{C})$. As in case of complex Yetter-Drinfel'd
modules, there exists a notion of a Radford biproduct or 
Majid bosonization for Hopf algebras in $\cC$ 
(see Definition \nref{d:RadfordBiproduct}); 
it turns a Hopf algebra $K$ in the braided category 
$\YD{A}(\mathcal{C})$ into a Hopf algebra 
$K\rtimes A$ in $\cC$.

The Radford projection theorem \nref{thm:RadfordProjection}
provides a converse: given a projection $\pi:H\rightarrow A$ in  $\cC$
to a Hopf subalgebra $A\subset H$, the  coinvariants 
$K\subset H$ with respect to $\pi$ have a natural structure
of a Hopf algebra  in the braided category $\YD{A}(\mathcal{C})$, 
such that $H$ can be expressed as a Radford biproduct,   
$H\cong  K\rtimes A$. It is then known
\cite[Proposition 4.2.3]{Besp95}, see also
Theorem \nref{thm:RadfordYDCategories}, that the following
braided categories are isomorphic:
$$\YD{K \rtimes A}(\cC) \cong \YD{K}(\YD{A}(\cC ))
\,\, . $$

We are now ready to describe the construction of partial dualization:
suppose that $A$ and $B$ are Hopf algebras and 
that $\omega : A\otimes B \rightarrow\unit_\cC$ is
a non-degenerate Hopf pairing. We relate the categories
of Yetter-Drinfel'd modules by an
isomorphism of braided categories
$$\Omega^\omega:\YD{A}(\cC)\stackrel\sim\to\YD{B}(\cC)
\,\, .$$
It is constructed in two steps: we use the Hopf pairing
$\omega$ to turn the left $A$-action into a
right $B$-coaction and the left $A$-coaction into
a right $B$-action. Then the braiding of $\cC$ is used in a second
step to turn right (co-)actions into left (co-)actions.
Schematically, indicating the relevant propositions of the
paper, we have
    \[
    \Omega^\omega:\;
    \begin{xy}
    \xymatrix{
    \YD{A}(\cC )\ar[r]^{\text{
    \nref{cor:pairing_functor}}} & \rYD{B}(\cC
)\ar[r]^{\text{\nref{thm:side_switch_functor}}} & \YD{B}(\cC ) .
    }
    \end{xy}
    \]
Diagrammatically, the $B$-action and $B$-coaction 
for the Yetter-Drinfel'd module $\Omega^\omega(X)$
are given as follows:  
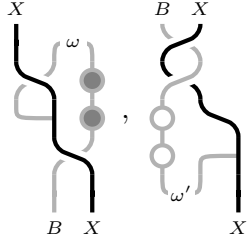
\begin{align*}
        \newcommand{\rhoomega}{\raisebox{-.5\totalheight}{
        \begin{tikzpicture}
                \begin{scope}[scale=0.5]
                        \dAction{0}{2}{1}{-1}{\grau}{black}
                        \dSkewantipode{2}{1}{\grau}
                        \dSkewantipode{2}{2}{\grau}
                        \dPairing{1}{3}{1}{1}{\grau}{\tiny $\! \omega \!$}
                        \vLine{0}{2}{1}{3}{\grau}
                        \vLineO{1}{2}{0}{3}{black}
                        \vLine{1}{0}{2}{1}{\grau}
                        \vLineO{2}{0}{1}{1}{black}
                        \vLine{0}{3}{0}{4}{black}
                        \vLine{1}{-1}{1}{0}{\grau}
                        \vLine{2}{-1}{2}{0}{black}
                \end{scope}
                \draw (0.5 , -0.7) node {\tiny $B$};
                \draw (1 , -0.7) node {\tiny $X$};
                \draw (0 , 2.2) node {\tiny $X$};
        \end{tikzpicture}
        }
        }
        \newcommand{\deltaomega}{\raisebox{-.5\totalheight}{
        \begin{tikzpicture}
                \begin{scope}[scale=0.5]
                        \dAction{1}{1}{1}{1}{\grau}{black}
                        \dAntipode{0}{1}{\grau}
                        \dAntipode{0}{2}{\grau}
                        \dCopairing{0}{1}{1}{1}{\grau}{\tiny $\!\! \omega' \!\!
\vspace*{-1mm}$}
                        \vLine{1}{3}{0}{4}{black}
                        \vLineO{0}{3}{1}{4}{\grau}
                        \vLine{1}{4}{0}{5}{\grau}
                        \vLineO{0}{4}{1}{5}{black}
                        \vLine{2}{0}{2}{1}{black}
                        \vLine{2}{2}{1}{3}{black}
                \end{scope}
                \draw (1 , -0.2) node {\tiny $X$};
                \draw (0 , 2.7) node {\tiny $B$};
                \draw (0.5 , 2.7) node {\tiny $X$};
        \end{tikzpicture}
        }
        }
        \newcommand{\omegatwo}{\raisebox{-.5\totalheight}{
        \begin{tikzpicture}
                \begin{scope}[scale=0.5]
                        \dAction{0}{2}{1}{1}{\grau}{black}
                        \dAction{1}{1}{1}{-1}{\grau}{black}
                        \vLine{0}{0}{0}{1}{black}
                        \vLine{0}{1}{1}{2}{black}
                        \vLineO{1}{1}{0}{2}{\grau}
                        \vLine{2}{1}{2}{3}{black}
                \end{scope}
                \draw (0 , -0.2) node {\tiny $X$};
                \draw (1 , -0.2) node {\tiny $Y$};
                \draw (0.5 , 1.7) node {\tiny $X$};
                \draw (1 , 1.7) node {\tiny $Y$};
        \end{tikzpicture}
        }
        }
         \rhoomega, \deltaomega 
    \end{align*}
Here, filled circles denote the inverse of the antipode and 
empty circles the antipode of $B$.

As already explained,
the input of our construction is a {\em partial dualization datum}
$\cA$:
it consists of a Hopf algebra projection  $\pi:\;H\rightarrow A$ 
to a Hopf subalgebra, and a Hopf algebra $B$ with a non-degenerate
Hopf pairing $\omega: A\otimes B\rightarrow \unit_\cC$.
In Section \nref{sec:PartialDualization}, we construct for a
given partial dualization datum $\cA$ a new Hopf algebra $r_\cA(H)$
in $\cC$ as follows:

\begin{enumerate}
\item The Radford projection theorem, applied to the
projection $\pi:H\to A$, allows us to write the Hopf algebra
$H$ in the form $H\cong K\rtimes A$, with $K$ a Hopf
algebra in the braided category $\YD{A}(\cC)$.

\item 
The braided monoidal equivalence $\Omega^\omega$ implies that
the image of the Hopf algebra $K$ in the braided category
$\YD{A}(\cC)$ is a Hopf algebra $L:=\Omega^\omega(K)$ in the braided category
$\YD{B}(\cC)$.

\item The partially dualized Hopf algebra $r_\cA(H)$ is
defined as 
the bosonization $r_\cA(H):=L\rtimes B$ of $L$. This is a Hopf
algebra in the braided category $\cC$.
\end{enumerate}

To summarize, we dualize a Hopf subalgebra $A$ of $H$
and at the same time covariantly transform the remaining 
coinvariants $K\subset H$ to $L\subset r_\cA(H)$.  As a combination
of contra- and covariant operations, 
partial dualization is {\em not} functorial in $H$.

We list some more results of Section \nref{sec:PartialDualization}:
\begin{itemize}
\item
{} the partial dualization datum $\cA$ for $H$, one can
obtain a partial dualization datum $\cA^-$ of the Hopf algebra
$r_\cA(H)$. There is a canonical isomorphism of Hopf algebras
in $\cC$ such that
    $$r_{\cA^-}\left(
    r_{\cA}(H)\right)\cong H \,\,, $$
showing that partial dualization is essentially involutive.

\item
Theorem \nref{thm:YetterDrinfeldsCoincide} then asserts
that the categories of Yetter-Drinfel'd modules for a Hopf algebra
$H$ in $\cC$ and its partial dualization $r_\cA(H)$ are braided
equivalent:    
      $$\YD{H}\left(\cC\right)
      \cong\YD{K}\left(\YD{A}\left(\cC\right)\right)
      \stackrel{\Omega^\omega}{\longrightarrow}
        \YD{L}\left(\YD{B}\left(\cC\right)\right)
      \cong\YD{r_\cA(H)}\left(\cC\right)$$

\end{itemize}

In Section \nref{sec:Examples}, we finally discuss three
classes of examples of partial dualizations.

%%%%%%%%%%%%%%%%%%%%%%%%%%%%%%%%%%%%%%%%%%%%%%%%%%%%%%%%%%%%%%%%%%%%%%%%%%%%%%%%
%%%%%%%%%%%%%%%%%%%%%%%%%%%%%
%%%%%%%%%%%%%%%%%%%%%%%%%%%%%%%%%%%%%%%%%%%%%%%%%%%%%%%%%%%%%%%%%%%%%%%%%%%%%%%%
%%%%%%%%%%%%%%%%%%%%%%%%%%%%%
%%%%%%%%%%%%%%%%%%%%%%%%%%%%%%%%%%%%%%%%%%%%%%%%%%%%%%%%%%%%%%%%%%%%%%%%%%%%%%%%
%%%%%%%%%%%%%%%%%%%%%%%%%%%%%  
%%%%%%%%%%%%%%%%%%%%%%%%%%%%%%%%%%%%%%%%%%%%%%%%%%%%%%%%%%%%%%%%%%%%%%%%%%%%%%%%
%%%%%%%%%%%%%%%%%%%%%%%%%%%%%
\section{Preliminaries}\label{sec:Preliminaries}
We assume that the reader is familiar with the definition of a braided 
monoidal category, see e.g. \cite{Kass95} as a general reference.
Denote by $\cC$ a  monoidal category with
tensor product $\otimes$ and unit object $\unit$;
without loss of generality, we assume that $\cC$ is strict.
If $\cC$ is endowed with a braiding, we denote it by $c : \otimes \rightarrow
\otimes^\op$.
For any braided category $\cC$,  
the monoidal category $\cC$ with the inverse braiding
$\ov{c}_{X,Y} := c^{-1}_{Y,X}$ is denoted by $\ov{\cC}$.\\
We use the graphical calculus for braided categories for
which we fix our conventions as follows:
diagrams are read from bottom to top. 
In Fig.\! \ref{morphisms_in_monoidal_categories} we depict the identity of an
object $X$ in $\cC$, a morphism $h : X_1 \otimes \ldots \otimes X_n \rightarrow
Y_1 \otimes \ldots \otimes Y_m$,
the composition $g \circ f$ of morphisms $f : X\rightarrow Y$ and $g : Y
\rightarrow Z$ and the tensor product of $f$ with
$f' : X' \rightarrow Y'$ by juxtaposition. 
The braiding $c_{X,Y} : X \otimes Y
\rightarrow Y \otimes X$ and the inverse braiding $c_{X,Y}^{-1} : Y \otimes X
\rightarrow X \otimes Y$ are shown in Fig.\!
\ref{braiding_and_Inverse_braiding}.
\begin{figure}[h]
  \centering
  $\id_X =$
  \begin{grform}
    \draw (0,0) -- (0,2);
    \draw (0,-0.5) node {$X$};
    \draw (0,2.5) node {$X$};
  \end{grform}, \quad
  $h =$
  \begin{tikzpicture}[intext]
    \draw (0,0) -- (0,2);
    \draw (1,0) -- (1,2);
    \draw (0,-0.5) node {$X_1$};
    \draw (1,-0.5) node {$X_n$};
    \draw (0,2.5) node {$Y_1$};
    \draw (1,2.5) node {$Y_m$};
    \draw (0.5,0.2) node {$\cdots$};
    \draw (0.5,1.8) node {$\cdots$};
    \draw (0.5,1) node [shape = rectangle, draw, fill = white] {$\quad h
\quad$};
  \end{tikzpicture},\quad
  $g \circ f =$
  \begin{tikzpicture}[intext]
    \draw (0 , 0) -- (0,2);
    \draw (0 , -0.5) node {$X$};
    \draw (0 , 2.5) node {$Z$};
    \draw (0.5 , 1) node {$Y$};
    \draw (0 , 0.5) node [shape = rectangle, draw, fill = white] {$ f $};
    \draw (0 , 1.5) node [shape = rectangle, draw, fill = white] {$ g $};
  \end{tikzpicture},\quad
  $f \otimes f' =$
  \begin{tikzpicture}[intext]
    \draw (0,0) -- (0,2);
    \draw (1,0) -- (1,2);
    \draw (0,-0.5) node {$X$};
    \draw (1,-0.5) node {$X'$};
    \draw (0,2.5) node {$Y$};
    \draw (1,2.5) node {$Y'$};
    \draw (0,1) node [shape = rectangle, draw, fill = white] {$ f $};
    \draw (1,1) node [shape = rectangle, draw, fill = white] {$ f' $};
  \end{tikzpicture}\quad.
  \caption{Graphical notation for morphisms in monoidal categories}
  \label{morphisms_in_monoidal_categories}
\end{figure}
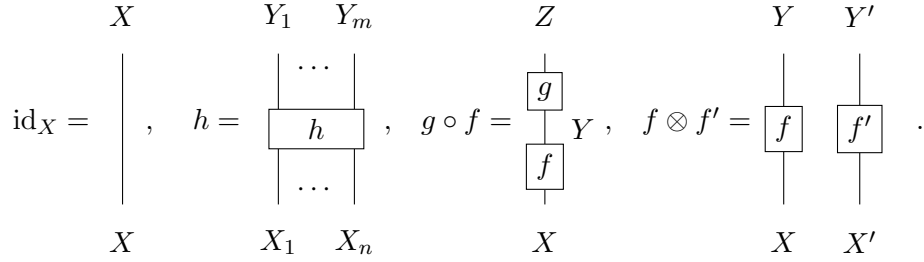
\begin{figure}
  \centering
  $c_{X,Y}= $
  \begin{tikzpicture}[intext]
    \draw (1,0) .. controls (1,1) and (0,1) .. (0,2);
    \draw [color = white, line width = 10pt] (0,0) .. controls (0,1) and (1,1)
.. (1,2);
    \draw (0,0) .. controls (0,1) and (1,1) .. (1,2);
    \draw (0,-0.5) node {$X$};
    \draw (1,-0.5) node {$Y$};
    \draw (0,2.5) node {$Y$};
    \draw (1,2.5) node {$X$};
  \end{tikzpicture},\quad
  $c_{X,Y}^{-1}= $
  \begin{tikzpicture}[intext]
    \draw (0,0) .. controls (0,1) and (1,1) .. (1,2);
    \draw [color = white, line width = 10pt] (1,0) .. controls (1,1) and (0,1)
.. (0,2);
    \draw (1,0) .. controls (1,1) and (0,1) .. (0,2);
    \draw (0,-0.5) node {$Y$};
    \draw (1,-0.5) node {$X$};
    \draw (0,2.5) node {$X$};
    \draw (1,2.5) node {$Y$};
  \end{tikzpicture}
  \caption{Braiding and inverse braiding}
  \label{braiding_and_Inverse_braiding}
\end{figure}
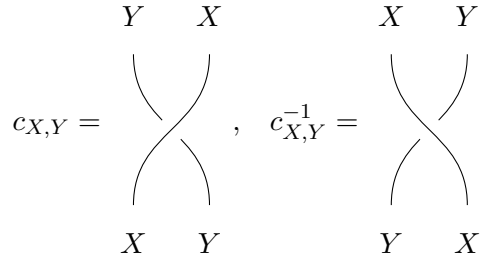

\subsection{Hopf algebras in braided categories}
We recall the definitions of an algebra and of a coalgebra
in a (not necessarily braided) monoidal category
and the notion of a bialgebra resp. Hopf algebra in a braided 
category $\cC$. For the braided category of $\Bbbk$-vector spaces,
these definitions specialize to the textbook definitions.

\begin{definition}
\label{d:Hopf_algebra}
        Let $\cC$ be a braided category. 
        An object $A$ together with morphisms 
        $\mu_A : A \otimes A \rightarrow A$, $\eta_A : \unit \rightarrow A$, 
        $\Delta_A : A \rightarrow A \otimes A$, $\eps_A : A \rightarrow \unit$
and $S : A \rightarrow A$ is called \emph{Hopf algebra in $\cC$}, 
        if
        \begin{enumerate}
                \item
                        the triple $(A, \mu_A, \eta_A)$ is a unital, associative
algebra in $\cC$, i.e.
                        \begin{align*}
                                \mu_A \circ (\mu_A \otimes \id_A ) &= \mu_A
\circ (\id_A \otimes \mu_A )\\
                                \mu_A \circ (\eta_A \otimes \id_A ) &= \id_A =
\mu_A \circ (\id_A \otimes \eta_A ),
                        \end{align*}
                \item
                        the triple $(A, \Delta_A, \eps_A)$ is a counital,
coassociative coalgebra in $\cC$, i.e.
                        \begin{align*}
                                (\Delta_A \otimes \id_A ) \circ \Delta &= (\id_A
\otimes \Delta_A ) \circ \Delta_A\\
                                (\eps_A \otimes \id_A ) \circ \Delta &= \id_A =
(\id_A \otimes \eps_A ) \circ \Delta_A,
                        \end{align*}
                \item
                        the morphisms $\mu_A, \eta_A, \Delta_A$ and $\eps_A$
obey the equations
                        \begin{align*}
                                \Delta_A \circ \mu_A  &= (\mu_A \otimes \mu_A )
\circ (\id_A \otimes c_{A,A} \otimes \id_A ) \circ (\Delta_A \otimes \Delta_A
)\\
                                \eps_A \circ \mu_A    &= \eps_A \otimes \eps_A\\
                                \Delta_A \circ \eta_A &= \eta_A \otimes \eta_A
\\
                                \eps_A \circ \eta_A   &= \id_{\unit},
                        \end{align*}
                \item
                        the morphism $S_A : A \rightarrow A$ is invertible and
obeys
                        \begin{align*}
                                \mu_A \circ (S_A \otimes \id_A ) \circ \Delta_A
= \eta_A \circ \eps_A = \mu_A \circ (\id_A \otimes S_A ) \circ \Delta_A.
                        \end{align*}
        \end{enumerate}
We call $\mu_A$ (resp. $\Delta_A$) the \emph{(co)multiplication} 
of $A$ and
$\eta_A$ (resp. $\eps_A$) the \emph{(co)unit} of $A$.
The morphism $S_A$ is called the \emph{antipode of $A$}.
\end{definition}

\begin{remark}
The unit, counit and antipode of a Hopf algebra are unique. 
%$(A,\mu,\eta, \Delta, \eps, S)$ are unique. 
        Thus, to define a Hopf algebra, it is only necessary to specify the
        multiplication and comultiplication and 
        we can unambiguously talk about the Hopf algebra 
        $(A,\mu, \Delta)$; sometimes, we suppress the structure
        morphisms in the notation.
\end{remark}

\begin{example}
\label{example:basic_Hopf_algebras}
\begin{enumerate}
\item
The monoidal unit $\unit$ of $\cC$ is a Hopf algebra with
all structural morphisms given by $\id_\unit$.
                \item
                        If $A = (A, \mu, \Delta)$ is a Hopf algebra in $\cC$, 
                        then $A^\op := (A, \mu^-, \Delta)$ and $A^\cop := (A, \mu, \Delta^-)$
                        with $\mu^- := \mu \circ c_{A,A}^{-1}$ and $\Delta := c_{A,A}^{-1} \circ \Delta$ 
                        are Hopf algebras not in $\cC$, but
                        rather in the category $\ov{\cC}$
                        with inverse braiding.
        \end{enumerate}
\end{example}

\begin{definition}
        Let $A$ and $B$ be Hopf algebras in $\cC$. A morphism $f : A \rightarrow
B$ in $\cC$ is a Hopf algebra morphism, 
        if $f$ is an algebra homomorphism,
        \[
                f \circ \mu_A = \mu_B \circ (f \otimes f) \quad \text{and} \quad
f \circ \eta_A = \eta_B,
        \]
        and a coalgebra homomorphism, 
        \[
                \Delta_B \circ f = (f \otimes f) \circ \Delta_A \quad \text{and}
\quad \eps_B \circ f = \eps_A.
        \]
\end{definition}

\begin{remark}
\label{rem:Antipode_as_homomorphism}
Standard textbook results continue to hold for Hopf algebras
in a braided category:
        Hopf algebra homomorphisms commute with the antipode, i.e. $f \circ S_A
= S_B \circ f$.
The antipode $S$ of a Hopf algebra $A$ is an isomorphism of Hopf
algebras in $\ov{\cC}$
        \[
                S : A^\op \rightarrow A^\cop.
        \]
Since $(A^\op)^\op = A$, 
we see that $S$ is also an isomorphism between
the following Hopf algebras in $\cC$
        \[
                S : A \rightarrow (A^\cop )^\op.
        \]
Note that $(A^\op )^\cop$ and $(A^\cop )^\op$ are in general different
Hopf algebras. 
        Nevertheless, $S^2$ is an isomorphism of Hopf algebras from $(A^\op
)^\cop$ to $(A^\cop )^\op$.
\end{remark}

\subsection{Modules and comodules over Hopf algebras}
In a monoidal category, modules over an associative algebra and comodules over an associative coalgebra 
are defined as usual. Modules, as well as comodules, over a Hopf algebra
in a braided category form a monoidal category. 
A new technical feature are
'side switch' functors $\sT$ which establish the equivalence of the
categories of left $A$-(co)mo\-du\-les in the braided category $\cC$ 
and right $A^\op$-modules (resp.\ $A^\cop$-comodules) in 
the braided category $\ov{\cC}$.
\begin{definition}
\quad
        \begin{enumerate}
\item
Let $A$ be an algebra in $\cC$. A \emph{left $A$-module}
is an object $X\in\cC$, together with a morphism 
$\rho = \rho_X : A \otimes X \rightarrow X$, such that
\begin{align*}
                                \rho \circ (\id_A \otimes \rho ) = \rho \circ
(\mu \otimes \id_X ) \quad \text{and} \quad \rho \circ (\eta \otimes \id ).
                        \end{align*}
                \item
                        Let $X$ and $Y$ be two $A$-modules. A morphism $f : X
\rightarrow Y$ is called \emph{$A$-linear}, if
                        \[
                                f \circ \rho_X = \rho_Y \circ (\id_A \otimes f).
                        \]
\item
Let $A$ be a coalgebra in $\cC$. A \emph{left $A$-comodule}
is  an object $Y$ in $\cC$,
together with a morphism 
$\delta = \delta_X : X \rightarrow A \otimes X$,
such that
                        \begin{align*}
                                (\id_A \otimes \delta ) \circ \delta = (\Delta
\otimes\id_X ) \circ \delta \quad \text{and} \quad ( \eps \otimes \id_X ) \circ
\delta_X = \id_X.
                        \end{align*}
                \item
                        Let $X$ and $Y$ be two $A$-comodules. A morphism $f : X
\rightarrow Y$ is called \emph{$A$-colinear}, if
                        \[
                                \delta_Y \circ f = (\id_A \otimes f) \circ
\delta_X.
                        \]
        \end{enumerate}
\end{definition}

\begin{remark}
\quad
\begin{enumerate}
\item
The left $A$-modules over a Hopf algebra $A$,
together with $A$-linear maps, form a monoidal category $\lMod{A}{\cC}$.
The tensor product of an $A$-module $(X, \rho_X)$ and an
$A$-module $(Y,\rho_Y)$ is given by
the usual action of $A$ on $X \otimes Y$, i.e.
$\rho_{X \otimes Y} := (\rho_X \otimes \rho_Y)
\circ (\id_A \otimes c_{A,X} \otimes \id_Y) 
\circ (\Delta \otimes \id_{X \otimes Y})$.
The monoidal unit is the $A$-module $(\unit, \eps)$.
\item
Similarly, left $A$-comodules over a
Hopf algebra $A$ form a monoidal category $\lComod{A}{\cC}$.
Given an $A$-comodule $(X, \delta_X)$ and
an $A$-comodule $(Y,\delta_Y)$ the
coaction of $A$ on $X \otimes Y$ is given by
$\delta_{X \otimes Y} := (\mu \otimes \id_{X
\otimes Y}) \circ (\id_A \otimes c_{X,A} \otimes \id_Y) \circ (\delta_X \otimes
\delta_Y)$.
The monoidal unit is the $A$-comodule $(\unit, \eta)$.
\item
The monoidal categories  of right $A$-modules and
right $A$-comodules are denoted by $\rMod{A}{\cC}$ and
$\rComod{A}{\cC}$, respectively.
\item
Figure \ref{fig:morphisms} lists our graphical
notation for structure morphism of Hopf algebras and 
left and right modules and comodules. Different colours
have the only purpose of improving the readability of the 
diagrams.
\end{enumerate}
\end{remark}

\begin{figure}[h]
  \centering
  $\mu = $
  \begin{tikzpicture}[intext]
    \begin{scope}[xscale = 0.5]
                \dMult{0.5}{0}{1}{1}{\grau}
    \end{scope}
    \draw (0.25,-0.3) node {$A$};
    \draw (0.75,-0.3) node {$A$};
    \draw (0.5,1.3) node {$A$};
  \end{tikzpicture},\quad
  $\eta = $
  \begin{tikzpicture}[intext]
    \draw[white] (0.5,0) -- (0.5,0.5);
    \begin{scope}[scale = 0.5]
                \vLine{1}{2}{1}{3}{\grau}
                \dUnit{1}{2}{\grau}
    \end{scope}
    \draw (0.5,1.8) node {$A$};
  \end{tikzpicture},\quad
  $\Delta = $
  \begin{tikzpicture}[intext]
    \begin{scope}[xscale = 0.5]
      \dMult{0.5}{1}{1}{-1}{\grau}
    \end{scope}
    \draw (0.5,-0.3) node {$A$};
    \draw (0.25,1.3) node {$A$};
    \draw (0.75,1.3) node {$A$};
  \end{tikzpicture},\quad
  $\eps = $
  \begin{tikzpicture}[intext]
    \draw[white] (0.5,1) -- (0.5,0.5);
    \begin{scope}[scale = 0.5]
                \dCounit{1}{0}{\grau}
                \vLine{1}{0}{1}{-1}{\grau}
    \end{scope}
    %\dCounit{0.5}{0}{\grau}
    \draw (0.5,-0.8) node {$A$};
  \end{tikzpicture},\quad
  $S = $
  \begin{tikzpicture}[intext]
    \begin{scope}[scale = 0.5]
                \vLine{1}{0}{1}{2}{\grau}
                \dAntipode{1}{0.5}{\grau}
    \end{scope}
    \draw (0.5,-0.3) node {$A$};
    \draw (0.5,1.3) node {$A$};
  \end{tikzpicture},\quad
  $S^{-1} = $
  \begin{tikzpicture}[intext]
    \begin{scope}[scale = 0.5]
                \vLine{1}{0}{1}{2}{\grau}
                \dSkewantipode{1}{0.5}{\grau}
    \end{scope}
    %\skewantipode{gray!60}
    \draw (0.5,-0.3) node {$A$};
    \draw (0.5,1.3) node {$A$};
  \end{tikzpicture},\\
  $\rho = $
  \begin{tikzpicture}[intext]
    \begin{scope}[scale = 0.5]
                \dAction{1}{0}{1}{2}{\grau}{black}
    \end{scope}
    \draw (0.5,-0.3) node {$A$};
    \draw (1,-0.3) node {$X$};
    \draw (1,1.3) node {$X$};
  \end{tikzpicture},\quad
  $\delta = $
  \begin{tikzpicture}[intext]
    \begin{scope}[scale = 0.5]
                \dAction{1}{2}{1}{-2}{\grau}{black}
    \end{scope}
    \draw (0.5,1.3) node {$A$};
    \draw (1,-0.3) node {$X$};
    \draw (1,1.3) node {$X$};
  \end{tikzpicture},
  $\rho^r = $
  \begin{tikzpicture}[intext]
    \begin{scope}[scale = 0.5]
                \dAction{2}{0}{-1}{2}{\grau}{black}
    \end{scope}
    \draw (1,-0.3) node {$A$};
    \draw (0.5,-0.3) node {$X$};
    \draw (0.5,1.3) node {$X$};
  \end{tikzpicture},\quad
  $\delta^r = $
  \begin{tikzpicture}[intext]
    \begin{scope}[scale = 0.5]
                \dAction{2}{2}{-1}{-2}{\grau}{black}
    \end{scope}
    \draw (1,1.3) node {$A$};
    \draw (0.5,-0.3) node {$X$};
    \draw (0.5,1.3) node {$X$};
  \end{tikzpicture}.
  \caption{Standard notation for structure morphisms}
  \label{fig:morphisms}
\end{figure}
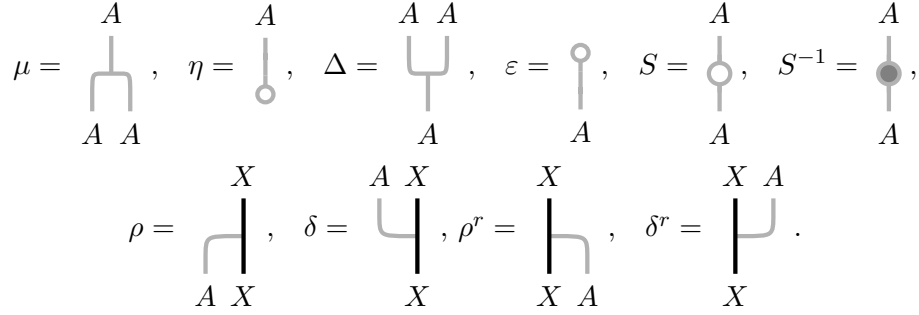

\begin{lemma}
For any  Hopf algebra homomorphism $\varphi : A \rightarrow B$, the
\emph{restriction} is  the strict monoidal functor 
$\us \circ (\varphi \otimes \id):
\lMod{B}{\cC} \rightarrow \lMod{A}{\cC}$, sending
the $B$-module $(Y, \rho)$ to the $A$-module $(Y,
\rho \circ (\varphi \otimes \id_Y))$. \\
{\em Corestriction} is strict monoidal functor 
$(\varphi \otimes \id) \circ
\us : \lComod{A}{\cC} \rightarrow \lComod{B}{\cC}$, sending the 
$A$-comodule $(X, \delta)$ to the $B$-comodule
$(X, (\varphi \otimes \id_X) \circ \delta )$.
\end{lemma}

The side switch functor uses the braiding on $\cC$ to turn
a left comodule into a right comodule:

\begin{lemma}
        \label{lem:twisting_functors}
Let $A$ be a Hopf algebra in $\cC$ and $(X,\delta)$
and $A$-comodule. The pair
        \[
                {}^A\sT (X, \delta) := (X, c^{-1}_{X,A} \circ \delta )
        \]
is a right $A^\cop$-comodule; the assignment ${}^A \sT$  defines
a strict 
monoidal functor ${}^A \sT : \lComod{A}{\cC} \rightarrow \rComod{{A^\cop}}{\ov{\cC}}$.
\end{lemma}

\begin{remark}
\label{rem:twisting_functors}
In the same way, the braiding induces strict monoidal functors
        \begin{align*}
                {}_A \sT :\; & \lMod{A}{\cC} \rightarrow \rMod{A^\op}{\ov{\cC}}, \quad \\
                \sT^A :\; & \rComod{A}{\cC} \rightarrow \lComod{A^\cop}{\ov{\cC}}, \quad \\
                \sT_A :\; & \rMod{A}{\cC} \rightarrow \lMod{A^\op}{\ov{\cC}}.
        \end{align*}
        Note that the functor $\sT^{A^\cop} : \rComod{A^\cop}{\ov{\cC}} \rightarrow
\lComod{(A^\cop)^\cop}{\ov{\ov{\cC}}} = \lComod{A}{\cC}$ is inverse to ${}^A \sT$.
        Similarly, the functors ${}_A \sT, \sT^A$ and $\sT_A$ are invertible.
\end{remark}

\subsection{Hopf pairings}
We finally turn to the definition of a Hopf pairing between two Hopf algebras in a braided
category $\cC$.
\begin{definition}
        Let $A$ and $B$ be Hopf algebras in $\cC$. 
        A morphism $\omega : A \otimes B \rightarrow \unit$ in 
        $\cC$ is called a
\emph{Hopf pairing}, if the following identities hold
        \begin{align*}
                \omega \circ (\mu_A \otimes \id_B) &= \omega \circ (\id_A
\otimes \omega \otimes \id_B) \circ (\id_{A \otimes A} \otimes \Delta_B ),\\
                \omega \circ (\eta_A \otimes \id_B) &= \eps_B,\\
                \omega \circ (\id_A \otimes \mu_B) &= \omega \circ (\id_A
\otimes \omega \otimes \id_B) \circ (\Delta_A \otimes \id_{B \otimes B} ),\\
                \omega \circ (\id_A \otimes \eta_B) &= \eps_A.
        \end{align*}
\end{definition}

\begin{remark}
\begin{enumerate}
\item
If $A$ has a right-dual object $A^\vee$ (see \cite[Ch.\ XIV]{Kass95}
for a definition), then $A^\vee$
has a natural structure of a Hopf algebra in $\cC$ such that the evaluation
morphism ${\rm ev} : A \otimes A^\vee \rightarrow \unit$ is a Hopf
pairing.
\item
A Hopf pairing $\omega$ relates antipodes, 
        \[
                \omega \circ (S_A \otimes \id_B ) = \omega \circ (\id_A \otimes
S_B ).
        \]
\end{enumerate}
\end{remark}

A Hopf pairing $\omega: A \otimes B \rightarrow \unit$ gives 
rise to a \emph{dualization functor} $^\omega\sD$ which relates 
modules and comodules of Hopf algebras in different categories:

\begin{lemma}
        \label{lem:duality_functor}
        Let $\omega : A \otimes B \rightarrow \unit$ be a Hopf pairing and $(X,
\delta)$ a left $B$-comodule and set
        \[
                {}^\omega \sD(X, \delta) := (X, (\omega \otimes \id_X ) \circ
(\id_A \otimes \delta )).
        \]
        The assignment ${}^\omega \sD(X, \delta)$ defines a strict monoidal
functor ${}^\omega \sD : \lComod{B}{\cC} \rightarrow \lMod{A^{\cop}}{\ov{\cC}}$.
\end{lemma}
\begin{proof}
Keeping in mind that $A$ and $A^\cop$ are equal as algebras,
the first two conditions on a Hopf pairing imply
that ${}^\omega\sD(X, \delta)$ is an $A^\cop$-module.
To see that the functor ${}^\omega\sD$ is strict monoidal, 
we note the equality of comodules ${}^\omega\sD(X) \otimes
{}^\omega\sD(Y)={}^\omega\sD(X \otimes Y)$ in $\ov{\cC}$ 
for any two $B$-comodules $X$ and $Y$, which immediately 
follows from the naturality of the braiding and the third
condition on a Hopf pairing. Finally,
the equality ${}^\omega\sD(X) \otimes
{}^\omega\sD(\unit) = {}^\omega\sD(X) = {}^\omega\sD(\unit) \otimes
{}^\omega\sD(X)$
follows from the fourth condition for a Hopf pairing.
\end{proof}

\begin{remark}
\quad
\label{example:HopfPairings}
                        If $\omega : A \otimes B \rightarrow \unit$ is a Hopf
pairing, then also 
                        \[
                                \omega^+ := \omega \circ c_{B,A} \circ (S_B
\otimes S_A) : B \otimes A \rightarrow \unit
                        \]
                        and
                        \[
                                \omega^- := \omega \circ c^{-1}_{A,B} \circ
(S^{-1}_B \otimes S^{-1}_A) : B \otimes A \rightarrow \unit
                        \]
                        are Hopf pairings.
\end{remark}

\begin{definition}
A Hopf pairing $\omega : A \otimes B \rightarrow \unit$ is called
\emph{non-degenerate}, if there is a morphism $\omega' : \unit \rightarrow B
\otimes A$, such that
        \[
                (\omega \otimes \id_A ) \circ (\id_A \otimes \omega') = \id_A
\quad \text{and} \quad (\id_B \otimes \omega) \circ (\omega' \otimes \id_B) =
\id_B.
        \]
\end{definition}

\begin{remark}
\begin{enumerate}
\item
If $A$ and $B$ are Hopf algebras over a field $\Bbbk$ that are related
by a non-degenerate Hopf pairing, then both $A$ and $B$ are finite-dimensional.

\item
        If $\omega : A \otimes B \rightarrow \unit$ is non-degenerate, the
morphism $\omega' : \unit \rightarrow B \otimes A$ is unique.
        We call $\omega'$ the \emph{inverse copairing} of $\omega$.\\
        The inverse copairing of a Hopf pairing is a \emph{Hopf copairing}, i.e.
the following axioms are fulfilled:
        \begin{align*}
                (\Delta_B \otimes \id_A ) \circ \omega' &= (\id_{B \otimes B}
\otimes \mu_A ) \circ (\id_B \otimes \omega' \otimes \id_A ) \circ \omega',\\
                (\eps_B \otimes \id_A ) \circ \omega' &= \eta_A,\\
                (\id_B \otimes \Delta_A ) \circ \omega' &= (\mu_B \otimes \id_{A
\otimes A} ) \circ (\id_B \otimes \omega' \otimes \id_A ) \circ \omega',\\
                (\id_B \otimes \eps_A ) \circ \omega' &= \eta_B.
        \end{align*}
If the Hopf pairing $\omega : A \otimes B \rightarrow \unit$ is 
non-degenerate, then the Hopf pairings
$\omega^+, \omega^- : B \otimes A \rightarrow \unit$ are non-degenerate
as well.
        The inverse copairing of $\omega^+$ is the morphism $(S^{-1} \otimes
S^{-1}) \circ c^{-1}_{A,B} \circ \omega'$, 
        the inverse copairing of $\omega^-$ is the morphism $(S \otimes S) \circ
c_{B,A} \circ \omega'$.
\end{enumerate}
\end{remark}

\begin{lemma}
\label{lem:dualization_modules}
If $\omega : A \otimes B \rightarrow \unit$ is a 
non-degenerate Hopf pairing, the strict monoidal functor 
        \[
                {}^\omega \sD : \lComod{B}{\cC} \rightarrow \lMod{A^\cop}{\ov{\cC}}
        \]
        from Lemma \ref{lem:duality_functor} is an isomorphism of categories.
\end{lemma}
\begin{proof}
Use the copairing $\omega'$ to
define a functor ${}_{\omega'}\sD : \lMod{A^\cop}{\ov{\cC}} \rightarrow
\lComod{B}{\cC}$ by sending the module $(X, \rho)$ to the $B$-comodule
        \[
                (X, (\id_B \otimes \rho ) \circ (\omega' \otimes \id_X )).
        \]
The properties of a Hopf copairing imply
that ${}_{\omega'}\sD$ is a functor. The relation between
$\omega$ and $\omega'$ imply that the functors
${}_{\omega'}\sD$ and ${}^\omega \sD$ are inverses.
\end{proof}

\section{Yetter-Drinfel'd modules in braided categories}\label{sec:YDM}
Yetter-Drinfel'd modules or crossed modules \cite{Mont93,Kass95}
for Hopf algebras over a field 
have been generalized in \cite{Besp95} for Hopf algebras in a braided
category $\cC$.
In this section, we show that the dualization functor 
${}^\omega\sD$ from Lemma \ref{lem:dualization_modules} 
associated to a non-degenerate Hopf pairing extends to a
strict monoidal functor between categories of
Yetter-Drinfel'd modules. Moreover, we combine
the side switch functors $\sT^A$ and $\sT_A$ for modules and
comodules from Lemma
\ref{rem:twisting_functors} into a (non-strict) 
side switch functor  
$\sT$ for  Yetter-Drinfel'd modules. 

\subsection{The Yetter-Drinfel'd condition}
A Yetter-Drin\-fel'd module is a module and a comodule,
subject to a compatibility condition.
Actions and coactions can be on the left or right; thus
there are four different types of Yetter-Drinfel'd modules.
Our main result can be understood in terms of Yetter-Drinfel'd
modules with left action and left coaction; the other categories only 
serve as a tool in the proofs.

\begin{definition}
Let $A$ be a Hopf algebra in a braided category $\cC$; 
suppose that $X$ is a left or right module and comodule over $A$. 
The corresponding Yetter-Drinfel'd conditions are
depicted in Figure \ref{fig:YD_conditions}.
\end{definition}
\begin{figure}[h]
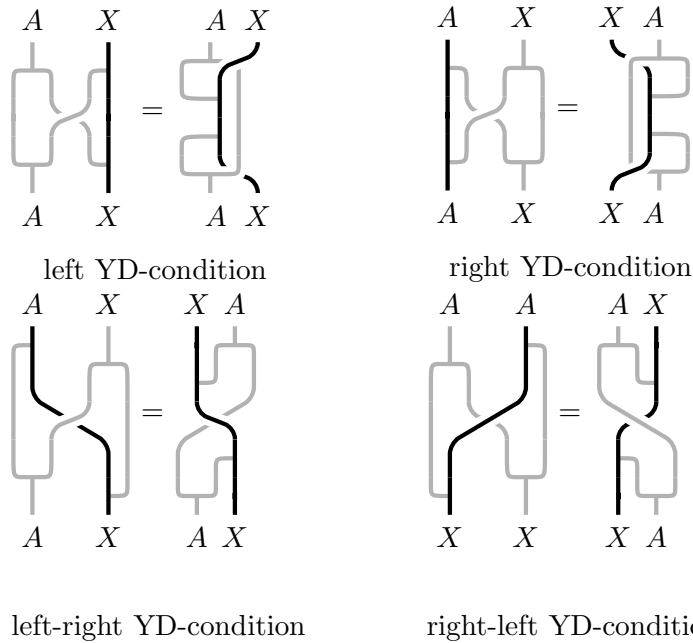

        %%%%%%%%%
        \begin{minipage}{0.3\textwidth}
                \begin{grform}
                        \begin{scope}[scale = 0.5]
                                \dMult{0.5}{1.5}{1}{-1.5}{\grau}
                                \dMult{0.5}{2.5}{1}{1.5}{\grau}
                                \dAction{2.5}{1.5}{0.5}{-1.5}{\grau}{black}
                                \dAction{2.5}{2.5}{0.5}{1.5}{\grau}{black}
                                \vLine{0.5}{1.5}{0.5}{2.5}{\grau}
                                \vLine{2.5}{1.5}{1.5}{2.5}{\grau}
                                \vLineO{1.5}{1.5}{2.5}{2.5}{\grau}
                                \vLine{3}{1.5}{3}{2.5}{black}
                        \end{scope}
                        \draw (0.5 , -0.3) node {$A$};
                        \draw (0.5 , 2.3) node {$A$};
                        \draw (1.5 , -0.3) node {$X$};
                        \draw (1.5 , 2.3) node {$X$};
                \end{grform}
                %%%%%%%%%
                =
                %%%%%%%%%
                \begin{grform}
                        \begin{scope}[scale = 0.5]
                                \vLine{2}{0}{1}{1}{black}
                                \dMultO{0}{1}{1.5}{-1}{\grau}
                                \dAction{0}{1}{1}{1}{\grau}{black}
                                \dAction{0}{3}{1}{-1}{\grau}{black}
                                \dMult{0}{3}{1.5}{1}{\grau}
                                \vLineO{2}{4}{1}{3}{black}
                                \vLine{1.5}{1}{1.5}{3}{\grau}
                        \end{scope}
                        \draw (0.45 , -0.3) node {$A$};
                        \draw (0.45 , 2.3) node {$A$};
                        \draw (1 , -0.3) node {$X$};
                        \draw (1 , 2.3) node {$X$};
                \end{grform}
                \\
\begin{center}
\text{left YD-condition}
\end{center}
\end{minipage}
\qquad\qquad
%%%%%%%%%
\begin{minipage}{0.3\textwidth}
\begin{grform}
\begin{scope}[scale = 0.5]
\dMult{1.5}{2.5}{1}{1.5}{\grau}
\dMult{1.5}{1.5}{1}{-1.5}{\grau}
\dAction{0.5}{1.5}{-0.5}{-1.5}{\grau}{black}
\dAction{0.5}{2.5}{-0.5}{1.5}{\grau}{black}
\vLine{2.5}{1.5}{2.5}{2.5}{\grau}
\vLine{1.5}{1.5}{0.5}{2.5}{\grau}
\vLineO{0.5}{1.5}{1.5}{2.5}{\grau}
\vLine{0}{1.5}{0}{2.5}{black}
\end{scope}
\draw (0 , -0.3) node {$A$};
\draw (0 , 2.3) node {$A$};
\draw (1 , -0.3) node {$X$};
\draw (1 , 2.3) node {$X$};
\end{grform}
%%%%%%%%%
=
%%%%%%%%%
\begin{grform}
\begin{scope}[scale = 0.5]
\vLine{1}{3}{0}{4}{black}
\dMultO{0.5}{3}{1.5}{1}{\grau}
\dAction{2}{3}{-1}{-1}{\grau}{black}
\dAction{2}{1}{-1}{1}{\grau}{black}
\dMult{0.5}{1}{1.5}{-1}{\grau}
\vLineO{0}{0}{1}{1}{black}
\vLine{0.5}{1}{0.5}{3}{\grau}
\end{scope}
\draw (0.55 , -0.3) node {$A$};
\draw (0.55 , 2.3) node {$A$};
\draw (0 , -0.3) node {$X$};
\draw (0 , 2.3) node {$X$};
\end{grform}
\\
\begin{center}
\text{right YD-condition}
\end{center}
\end{minipage}
%%%%%%%%%%%%%
\begin{minipage}{0.3\textwidth}
\begin{grform}
\begin{scope}[scale = 0.5]
\dMult{0.5}{2}{1}{-2}{\grau}
\dMult{2.5}{3}{1}{2}{\grau}
\dAction{0.5}{4}{0.5}{1}{\grau}{black}
\dAction{3.5}{1}{-0.5}{-1}{\grau}{black}
\vLine{3}{1}{1}{4}{black}
\vLine{0.5}{2}{0.5}{4}{\grau}
\vLine{3.5}{1}{3.5}{3}{\grau}
\vLineO{1.5}{2}{2.5}{3}{\grau}
\end{scope}
\draw (0.5 , -0.3) node {$A$};
\draw (0.5 , 2.8) node {$A$};
\draw (1.5 , -0.3) node {$X$};
\draw (1.5 , 2.8) node {$X$};
\end{grform}
%%%%%%%%%
=
%%%%%%%%%
\begin{grform}
\begin{scope}[scale = 0.5]
\vLine{0.5}{1}{2.5}{4}{\grau}
\dMult{0.5}{1}{1}{-1}{\grau}
\dMult{1.5}{4}{1}{1}{\grau}
\dAction{1.5}{1}{0.5}{1}{\grau}{black}
\dAction{1.5}{4}{-0.5}{-1}{\grau}{black}
\vLineO{2}{2}{1}{3}{black}
\vLine{2}{0}{2}{1}{black}
\vLine{1}{4}{1}{5}{black}
\end{scope}
\draw (0.5 , -0.3) node {$A$};
\draw (0.5 , 2.8) node {$X$};
\draw (1 , -0.3) node {$X$};
\draw (1 , 2.8) node {$A$};
                \end{grform}
                \\
                \begin{center}
                        \text{left-right YD-condition}
                \end{center}
        \end{minipage}
        %%%%%%%%%
        \qquad\qquad
        %%%%%%%%%
        \begin{minipage}{0.3\textwidth}
                \begin{grform}
                        \begin{scope}[scale = 0.5]
                                \dMult{0.5}{3}{1}{2}{\grau}
                                \dMult{2.5}{2}{1}{-2}{\grau}
                                \dAction{0.5}{1}{0.5}{-1}{\grau}{black}
                                \dAction{3.5}{4}{-0.5}{1}{\grau}{black}
                                \vLine{2.5}{2}{1.5}{3}{\grau}
                                \vLine{0.5}{1}{0.5}{3}{\grau}
                                \vLine{3.5}{2}{3.5}{4}{\grau}
                                \vLineO{1}{1}{3}{4}{black}
                        \end{scope}
                        \draw (0.5 , -0.3) node {$X$};
                        \draw (0.5 , 2.8) node {$A$};
                        \draw (1.5 , -0.3) node {$X$};
                        \draw (1.5 , 2.8) node {$A$};
                \end{grform}
                %%%%%%%%%
                =
                %%%%%%%%%
                \begin{grform}
                        \begin{scope}[scale = 0.5]
                                \vLine{1}{2}{2}{3}{black}
                                \dMult{1.5}{1}{1}{-1}{\grau}
                                \dMult{0.5}{4}{1}{1}{\grau}
                                \dAction{1.5}{1}{-0.5}{1}{\grau}{black}
                                \dAction{1.5}{4}{0.5}{-1}{\grau}{black}
                                \vLineO{2.5}{1}{0.5}{4}{\grau}
                                \vLine{1}{0}{1}{1}{black}
                                \vLine{2}{4}{2}{5}{black}
                        \end{scope}
                        \draw (0.5 , -0.3) node {$X$};
                        \draw (0.5 , 2.8) node {$A$};
                        \draw (1 , -0.3) node {$A$};
                        \draw (1 , 2.8) node {$X$};
                \end{grform}
                \\
                \begin{center}
                        \text{right-left YD-condition}
                \end{center}
        \end{minipage}
        %%%%%%%%%
        \caption{Yetter-Drinfel'd conditions}
        \label{fig:YD_conditions}
\end{figure}

The (left) Yetter-Drinfel'd modules over $A$ are objects of
a category $\YD{A}(\cC)$; morphisms in $\YD{A}(\cC)$ are 
morphisms in $\cC$ that are $A$-linear and $A$-colinear.\\
The tensor product of a Yetter-Drinfel'd module $X$ and a
Yetter-Drinfel'd module $Y$ is given by the object $X \otimes Y$
with the obvious action and coaction. The unit object is the 
monoidal unit $\unit$ of $\cC$, together with trivial
action given by the counit and trivial coaction given by
the unit of $A$.\\
        The braiding isomorphism 
        \[
                c^{\YD{}}_{X,Y} : X \otimes Y \rightarrow Y \otimes X
        \]
        is given by
        \begin{align}
        \label{YD-braiding}
                c^{\YD{}}_{X,Y} := (\rho_Y \otimes \id_X ) \circ (\id_A \otimes
c_{X,Y} ) \circ (\delta_X \otimes \id_Y )\,\, ;
        \end{align}
its inverse is
        \[
                (c^{\YD{}}_{X,Y})^{-1} := c^{-1}_{X,Y} \circ (\rho_Y \otimes
\id_X ) \circ (c^{-1}_{A,Y} \otimes \id_X )
                \circ (\id_Y \otimes S^{-1} \otimes \id_X ) \circ (\id_Y \otimes
\delta_X ).
\]

The structure is summarized in the following proposition
whose proof can be found in \cite{Besp95}.
\begin{proposition}
Let $A$ be a Hopf algebra in $\cC$. 
The left Yetter-Drinfel'd modules over $A$ in $\cC$ have
a natural structure of a
braided monoidal category $\YD{A}(\cC)$.
\end{proposition}

\begin{remark}
The definition of Yetter-Drinfel'd module does not require
the existence of an antipode so that Yetter-Drinfel'd modules
can be defined over a bialgebra as well.

If $A$ is a Hopf algebra, the antipode allows us to
reformulate the Yetter-Drinfel'd condition: a graphical calculation shows
that a module and comodule $X$ is a left Yetter-Drinfel'd module,
iff
        
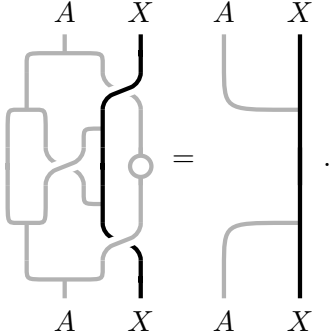
\[
                \begin{grform}
                        \begin{scope}[scale = 0.5]
                                \dMult{1}{1}{2}{-1}{\grau}
                                \dMult{0.5}{3}{1}{-2}{\grau}
                                \dMult{1}{6}{2}{1}{\grau}
                                \dMult{0.5}{4}{1}{2}{\grau}
                                \dAction{2.5}{3}{0.5}{-1}{\grau}{black}
                                \dAction{2.5}{4}{0.5}{1}{\grau}{black}
                                \vLine{2.5}{3}{1.5}{4}{\grau}
                                \vLineO{1.5}{3}{2.5}{4}{\grau}
                                \vLine{4}{1}{3}{2}{black}
                                \vLineO{3}{1}{4}{2}{\grau}
                                \vLine{4}{5}{3}{6}{\grau}
                                \vLineO{3}{5}{4}{6}{black}
                                \vLine{0.5}{3}{0.5}{4}{\grau}
                                \vLine{3}{3}{3}{4}{black}
                                \vLine{4}{2}{4}{5}{\grau}
                                \vLine{4}{0}{4}{1}{black}
                                \vLine{4}{6}{4}{7}{black}
                                \dAntipode{4}{3}{\grau}
                        \end{scope}
                        \draw (1 , -0.3) node {$A$};
                        \draw (2 , -0.3) node {$X$};
                        \draw (1 , 3.8) node {$A$};
                        \draw (2 , 3.8) node {$X$};
                \end{grform}
                =
                \begin{grform}
                        \begin{scope}[scale = 0.5]
                                \dAction{0}{0}{2}{4}{\grau}{black}
                                \dAction{0}{7}{2}{-4}{\grau}{black}
                        \end{scope}
                        \draw (0 , -0.3) node {$A$};
                        \draw (1 , -0.3) node {$X$};
                        \draw (0 , 3.8) node {$A$};
                        \draw (1 , 3.8) node {$X$};
                \end{grform}
                .
        \]
\end{remark}
This reformulation is useful to prove the following lemma
which is proven by straightforward calculations:        

\begin{lemma}
\label{lem:properties_of_theta}
Let $A$ be a Hopf algebra in $\cC$. For a  
left Yetter-Drinfel'd $X$ consider
        \[
\theta_X := \rho_X \circ (S \otimes \id ) \circ \delta_X
\in\mathrm{End}_\cC(X).
        \]
        The following holds
        \begin{enumerate}
                \item
                        $\theta_X \circ \rho_X = \rho_X \circ c_{X,A} \circ
c_{A,X} \circ (S^2 \otimes \id)$.
                \item
                        $\delta_X \circ \theta_X = (S^2 \otimes \id) \circ
c_{X,A} \circ c_{A,X} \circ \delta_X$.
                \item
                        If $Y$ is another Yetter-Drinfel'd module, we have
                        \[
                                c^{\YD{}}_{Y,X} \circ \theta_{Y \otimes X} \circ
c^{\YD{}}_{X,Y} =
                                c_{Y,X} \circ (\theta_Y \otimes \theta_X ) \circ
c_{X,Y} \; .
                        \]
        \end{enumerate}
\end{lemma}

\begin{remark}
The right Yetter-Drinfel'd modules also form a braided mo\-no\-i\-dal
category, which is denoted by $\rYD{A}(\cC)$. The braiding is given by
        \[
                c^{\YD{}}_{X,Y} := (\id_Y \otimes \rho^r_X ) \circ (c_{X,Y}
\otimes \id_A ) \circ (\id_X \otimes \delta^r_Y ).
        \]
If $\cC$ is the category of vector spaces over a field $\Bbbk$, 
we also
write $\YD{A}_\Bbbk$ or $\YD{A}$ for the category of Yetter-Drinfel'd modules.
\end{remark}

\begin{remark}
The collection of left-right Yetter-Drinfel'd modules forms 
a category $\lrYD{A}{A}(\cC)$. In contrast to the categories 
$\YD{A}(\cC)$ and $\rYD{A}(\cC)$, the category 
$\lrYD{A}{A}(\cC)$ can be endowed with two different 
tensor products (leading to monoidally equivalent categories) 
\begin{enumerate}
\item
For the first product, $A$ acts as usual 
on the product of two Yetter-Drinfel'd modules, while the 
coaction is given by
        \[
                (\id_{X \otimes Y} \otimes \mu ) \circ (\id_X \otimes \delta^r_Y
\otimes \id_A ) \circ (\id_X \otimes c^{-1}_{Y,A}) \circ (\delta^r_X \otimes
\id_Y ). 
        \]
Since this is the diagonal coaction of the Hopf algebra $A^\op$ in
$\ov{\cC}$, we denote left-right
Yetter-Drinfel'd modules with this tensor product 
by $\lrYD{A}{A^\op}(\cC)$.
The monoidal category $\lrYD{A}{A^\op}(\cC)$ admits the braiding
        \[
                c^{\YD{}}_{X,Y} := (\id_Y \otimes \rho_X ) \circ (\delta^r_Y
\otimes \id_X ) \circ c^{-1}_{Y,X}.
        \]
\item
A second tensor product on $\lrYD{A}{A}(\cC)$ is given by 
the diagonal action of $A^\cop$ and the diagonal coaction of 
$A$ on tensor products of Yetter-Drinfel'd modules.
We denote this monoidal category by $\lrYD{A^\cop}{A}(\cC)$;
it admits the braiding
        \[
                c^{\YD{}}_{X,Y} := c^{-1}_{Y,X} \circ (\id_X \otimes \rho_Y )
\circ (\delta^r_X \otimes \id_Y ).
        \]
\end{enumerate}
The two braided categories $\rlYD{A^\cop}{A}(\cC)$ and
$\rlYD{A}{A^\op}(\cC)$ are defined in complete analogy.
\end{remark}

\subsection{Radford biproduct and projection theorem}

The following situation is standard: let $A$ be a Hopf
algebra over a field $\Bbbk$. Let $K$ be a
Hopf algebra in the braided category $\YD{A}$ of
$A$-Yetter-Drinfel'd modules.

The category of Yetter-Drinfel'd modules over $K$ in 
$\YD{A}$ can be described
as the category of Yetter-Drinfel'd modules over a 
Hopf algebra $K \rtimes A$ over the field $\Bbbk$.
The Hopf algebra $K \rtimes A$ is called \emph{Majid bosonization} or
\emph{Radford's biproduct}. 
The definition of the biproduct $K \rtimes A$ directly generalizes
to the description of Yetter-Drinfel'd modules over a Hopf 
algebra $K$ in the braided category $\YD{A}(\cC)$, where $\cC$ is 
now an arbitrary braided category.
We collect in this subsection results from \cite{Besp95} 
that we will be needed in the
construction of the partially dualized Hopf algebra in Section
\ref{sec:PartialDualization}.

\begin{definition}[Radford Biproduct]
\label{d:RadfordBiproduct}
Let $\cC$ be a braided category and
let $A\in\cC$ and $K\in\YD{A}(\cC)$ be Hopf algebras. The \emph{Radford biproduct} 
$K \rtimes A$ is defined as the object $K \otimes A$ in $\cC$
together with the following morphisms:
\[
\mu_{K \rtimes A} :=
%%%%%%%%%%%%%%%%%%%
\begin{grform}
\begin{scope}[scale = 0.5]
\dMult{0.5}{3}{2}{1}{black}
\dMult{3.5}{3}{1}{1}{\grau}
\dMult{1.5}{1}{1}{-1}{\grau}
\dAction{1.5}{2}{1}{1}{\grau}{black}
\vLine{0.5}{0}{0.5}{3}{black}
\vLine{3.5}{1}{2.5}{2}{black}
\vLine{3.5}{0}{3.5}{1}{black}
\vLineO{2.5}{1}{3.5}{2}{\grau}
\vLine{3.5}{2}{3.5}{3}{\grau}
\vLine{1.5}{1}{1.5}{2}{\grau}
\vLine{4.5}{0}{4.5}{3}{\grau}
\end{scope}
\draw (0.25 , -0.3) node {$K$};
\draw (1 , -0.3) node {$A$};
\draw (1.75 , -0.3) node {$K$};
\draw (2.25 , -0.3) node {$A$};
\end{grform}
%%%%%%%%%%%%%%%%%%%
, \quad \Delta_{K \rtimes A} :=
%%%%%%%%%%%%%%%%%%%
\begin{grform}
\begin{scope}[scale = 0.5]
\dMult{0.5}{1}{2}{-1}{black}
\dMult{3.5}{1}{1}{-1}{\grau}
\dMult{1.5}{3}{1}{1}{\grau}
\dAction{1.5}{2}{1}{-1}{\grau}{black}
\vLine{0.5}{4}{0.5}{1}{black}
\vLine{2.5}{3}{3.5}{2}{\grau}
\vLineO{3.5}{3}{2.5}{2}{black}
\vLine{3.5}{4}{3.5}{3}{black}
\vLine{3.5}{2}{3.5}{1}{\grau}
\vLine{1.5}{3}{1.5}{2}{\grau}
\vLine{4.5}{4}{4.5}{1}{\grau}
\end{scope}
\draw (0.75 , -0.3) node {$K$};
\draw (2 , -0.3) node {$A$};
\end{grform}
%%%%%%%%%%%%%%%%%%%
, \quad S_{K \rtimes A}
%%%%%%%%%%%%%%%%%%%
\begin{grform}
\begin{scope}[scale = 0.5]
\dMult{0}{2}{1}{1}{\grau}
\dAction{0}{1}{1}{-1}{\grau}{black}
\vLine{0}{1}{0}{2}{\grau}
\vLine{2}{0}{2}{1}{\grau}
\vLine{2}{2}{2}{3}{black}
\vLine{2}{1}{1}{2}{\grau}
\vLineO{1}{1}{2}{2}{black}
%%%%
\dMult{0}{5}{1}{-1}{\grau}
\dAction{0}{6}{1}{1}{\grau}{black}
\vLine{0}{6}{0}{5}{\grau}
\vLine{2}{7}{2}{6}{\grau}
\vLine{2}{5}{2}{4}{black}
\vLine{1}{6}{2}{5}{black}
\vLineO{2}{6}{1}{5}{\grau}
%%%%
\dAntipode{0.5}{3}{\grau}
\dAntipode{2}{3}{black}
\end{scope}
\draw (0.5 , -0.3) node {$K$};
\draw (1 , -0.3) node {$A$};
\end{grform}.
%%%%%%%%%%%%%%%%%%%
\]
\end{definition}

\begin{proposition}
\label{prop:RadfordBiproduct}
The Radford biproduct $K \rtimes A$ is a Hopf algebra in $\cC$.
\end{proposition}

Definition \ref{d:RadfordBiproduct} and Proposition \ref{prop:RadfordBiproduct} are found in \cite[Subsection 4.1]{Besp95}.

\begin{remark}
If $K$ is a Hopf algebra in the category $\YD{A}$ of Yetter-Drinfel'd modules over a $\Bbbk$-Hopf algebra $A$,
the Radford biproduct is given by the following formulas for multiplication and
comultiplication, cf. \cite[Section 10.6]{Mont93}:
\begin{align*}
(h\otimes a)\cdot (k \otimes b) &= h\cdot (a_{(1)}.k) \otimes a_{(2)}\cdot b\\
\Delta(h \otimes a) &= h_{(1)} \otimes (h_{(2)})_{(-1)}\cdot a_{(1)} \otimes (h_{(2)})_{(0)} \otimes a_{(2)}.
\end{align*}
This is a special case of the formulas expressed graphically
in Definition \ref{d:RadfordBiproduct}.
If $K$ has only the structure of an algebra in $\YD{A}$, 
the vector space $K \otimes A$ with the multiplication 
$\mu_{K \rtimes A}$
is called a \emph{smash-product}.
If $K$ has only the structure of a coalgebra in $\YD{A}$, 
$K \otimes A$ with the comultiplication $\Delta_{K \rtimes A}$
is called a \emph{cosmash-product}.
\end{remark}

\begin{theorem}[Radford projection theorem]
\label{thm:RadfordProjection}
Let $H$ and $A$ be Hopf algebras in a braided category $\cC$. 
Let $\pi : H \rightarrow A$ and
$\iota : A \rightarrow H$ be Hopf algebra morphisms 
such that $\pi \circ \iota = \id_A$.
If $\cC$ has equalizers and $A \otimes \us$ preserves 
equalizers, there is a
Hopf algebra $K$ in the braided category $\YD{A}(\cC)$, such that
\[
H = K \rtimes A.
\]
\end{theorem}

\begin{proof}
For a complete proof we refer to \cite{AF00}.
\end{proof}

\begin{remark}
\label{rem:projection_thm}
To illustrate the situation, we discuss the case when
$\cC$ is the braided category of $\Bbbk$-vector spaces
and $\pi : H \rightarrow A$ is a projection to a Hopf subalgebra $A \subset
H$:\\
The vector space underlying the Hopf algebra $K$ in $\YD{A}$ is then
the space of coinvariants of $H$:
\[
K := H^{\rm{coin}(\pi)} := \left\{ r \in H \mid r_{(1)} \otimes \pi(r_{(2)}) = r \otimes 1 \right\} \,\, .
\]
One easily checks that $K$ is a subalgebra of $H$ and $K$ is invariant under
the left adjoint action of $A$ on $H$.\\
The subspace $K$ is also a left $A$-comodule with 
coaction $\delta_K(k):= \pi(k_{(1)}) \otimes k_{(2)}$. 
The fact that $H$ is a left $H$-Yetter-Drinfel'd module with the 
adjoint action and regular coaction implies that
$K$ is even an $A$-Yetter-Drinfel'd module.
The comultiplication of $K$ is given by the formula 
\[
\Delta_K(k) := k_{(1)}\pi(S_H(k_{(2)})) \otimes  k_{(3)}
\] 
and the antipode is $S_K(k) = \pi(k_{(1)})S_H(k_{(2)})$.
\end{remark}

\begin{theorem}[Bosonization Theorem]
\label{thm:RadfordYDCategories}
Let $A$ be a Hopf algebra in $\cC$ and $K$ a Hopf algebra in $\YD{A}(\cC)$.
There is an obvious
isomorphism of braided categories
\[
\YD{K \rtimes A}(\cC) \cong \YD{K}(\YD{A}(\cC )).
\]
\end{theorem}

For a proof, we refer to \cite[Proposition 4.2.3]{Besp95}.

\subsection{Equivalence of categories of left and right modules}
In this subsection we discuss the side switch functor $\sT :
\rYD{A}(\cC) \rightarrow \YD{A}(\cC)$ for Yetter-Drinfel'd
modules. It turns out that, for our purposes,
a non-trivial monoidal structure 
$\sT_2 : \sT \otimes \sT \rightarrow \sT \circ
\otimes$ has
to be chosen for the switch functor, even in those cases
(for $\cC$ symmetric) where the
identities provide a monoidal structure on $\sT$.

\begin{lemma}
\label{lem:partial_side_switch_functor}
The isomorphism ${}^A\sT : \lComod{A}{\cC} \rightarrow \rComod{\ov{\cC}}{A^\cop}$ of
categories from 
        Lemma \ref{lem:twisting_functors} extends to an isomorphism of
categories
        \[
                {}^A\sT : \YD{A}(\cC) \rightarrow
\lrYD{A^\cop}{A^\cop}(\ov{\cC}).
        \]
The functor ${}^A\sT$ is braided and strict monoidal, 
considered as a functor between the following monoidal
categories:
        \[
                {}^A\sT : \YD{A}(\cC) \rightarrow
\lrYD{(A^\cop)^\cop}{A^\cop}(\ov{\cC}).
        \]
\end{lemma}

\begin{remark}
The equality $(A^\cop)^\cop = A$ of Hopf algebras from Remark
\ref{rem:Antipode_as_homomorphism} might suggest the notation
\[
\lrYD{A}{A^\cop}(\ov{\cC}) := \lrYD{(A^\cop)^\cop}{A^\cop}(\ov{\cC})
\]
which is not in conflict with other notation used in this
article. To avoid confusion with the
different monoidal category $\lrYD{A}{A^\op}(\cC)$,
we refrain from using this notation.
\end{remark}

\begin{proof}
        Let $X=(X, \rho, \delta)$ be in $\YD{A}(\cC)$. 
        It follows from Lemma \ref{lem:twisting_functors} that ${}^A\sT(X)=(X,
\rho, c^{-1}_{X,A} \circ \delta)$ is an 
$A^{\cop}$-comodule and $A^\cop$-module in $\ov{\cC}$.
It remains to be shown that  ${}^A\sT(X)$ obeys the
condition of a left-right
$A^\cop$-Yetter-Drinfel'd module in $\ov{\cC}$:
        \begin{align*}
        %%%%%%%%%%%%%%%%%%%
                \begin{grform}
                        \begin{scope}[scale = 0.5]
                                \vLine{0}{0}{0}{3}{white}
                                \dMult{0}{2}{1}{-2}{\grau}
                                \dMult{2}{4}{1}{3}{\grau}
                                \dAction{2}{2}{1}{-2}{\grau}{black}
                                \dAction{0}{4}{1}{3}{\grau}{black}
                                \vLine{0}{2}{1}{3}{\grau}
                                \vLineO{1}{2}{0}{3}{\grau}
                                \vLine{2}{2}{3}{3}{\grau}
                                \vLineO{3}{2}{2}{3}{black}
                                \vLine{1}{3}{2}{4}{\grau}
                                \vLineO{2}{3}{1}{4}{black}
                                \vLine{0}{3}{0}{4}{\grau}
                                \vLine{3}{3}{3}{4}{\grau}
                        \end{scope}
                        \draw (0.25 , -0.3) node {$A$};
                        \draw (1.5 , -0.3) node {$X$};
                \end{grform}
                %%%%%%%%%%%%%%%%%%%
                =
                %%%%%%%%%%%%%%%%%%%
                \begin{grform}
                        \begin{scope}[scale = 0.5]
                                \vLine{0}{0}{0}{3}{white}
                                \dMult{0.5}{2}{1}{-2}{\grau}
                                \dMult{0.5}{3}{1}{2}{\grau}
                                \dAction{2.5}{2}{0.5}{-2}{\grau}{black}
                                \dAction{2.5}{3}{0.5}{2}{\grau}{black}
                                \vLine{2.5}{2}{1.5}{3}{\grau}
                                \vLineO{1.5}{2}{2.5}{3}{\grau}
                                \vLine{1}{5}{3}{7}{\grau}
                                \vLineO{3}{5}{1}{7}{black}
                                \vLine{3}{2}{3}{3}{black}
                                \vLine{0.5}{2}{0.5}{3}{\grau}
                        \end{scope}
                        \draw (0.5 , -0.3) node {$A$};
                        \draw (1.5 , -0.3) node {$X$};
                \end{grform}
                %%%%%%%%%%%%%%%%%%%
                =
                %%%%%%%%%%%%%%%%%%%
                \begin{grform}
                        \begin{scope}[scale = 0.5]
                                \vLine{0}{0}{0}{3}{white}
                                \dMult{0}{1}{1}{-1}{\grau}
                                \dMult{0}{5}{1}{1}{\grau}
                                \dAction{0}{2}{1}{1}{\grau}{black}
                                \dAction{0}{4}{1}{-1}{\grau}{black}
                                \vLine{2}{1}{1}{2}{black}
                                \vLineO{1}{1}{2}{2}{\grau}
                                \vLine{2}{4}{1}{5}{\grau}
                                \vLineO{1}{4}{2}{5}{black}
                                \vLine{0.5}{6}{2}{7}{\grau}
                                \vLineO{2}{6}{0.5}{7}{black}
                                \vLine{0}{1}{0}{2}{\grau}
                                \vLine{2}{0}{2}{1}{black}
                                \vLine{2}{2}{2}{4}{\grau}
                                \vLine{0}{4}{0}{5}{\grau}
                                \vLine{2}{5}{2}{6}{black}
                        \end{scope}
                        \draw (0.25 , -0.3) node {$A$};
                        \draw (1 , -0.3) node {$X$};
                \end{grform}
                %%%%%%%%%%%%%%%%%%%
                =
                %%%%%%%%%%%%%%%%%%%
                \begin{grform}
                        \begin{scope}[scale = 0.5]
                                \vLine{0}{0}{0}{3}{white}
                                \dMult{0}{1}{1}{-1}{\grau}
                                \dMult{1}{6}{1}{1}{\grau}
                                \dAction{1}{2}{1}{1}{\grau}{black}
                                \dAction{0}{5}{1}{-1}{\grau}{black}
                                \vLine{0}{1}{1}{2}{\grau}
                                \vLineO{1}{1}{0}{2}{\grau}
                                \vLine{2}{3}{1}{4}{black}
                                \vLineO{0}{2}{2}{4}{\grau}
                                \vLine{0}{5}{1}{6}{\grau}
                                \vLineO{1}{5}{0}{6}{black}
                                \vLine{2}{0}{2}{2}{black}
                                \vLine{2}{4}{2}{6}{\grau}
                                \vLine{0}{6}{0}{7}{black}
                        \end{scope}
                        \draw (0.25 , -0.3) node {$A$};
                        \draw (1 , -0.3) node {$X$};
                \end{grform}.
                %%%%%%%%%%%%%%%%%%%
        \end{align*}
One finally verifies that the braiding isomorphisms 
in the categories $\YD{A}(\cC)$ and 
$\lrYD{(A^\cop)^\cop}{A^\cop}(\ov{\cC})$ of Yetter-Drinfel'd
modules coincide as morphisms in the underlying
category $\cC$.
\end{proof}

\begin{remark}
One can show by similar arguments that the isomorphisms ${}_A\sT, \sT^A$ and
$\sT_A$ extend to braided and strict monoidal functors
        \begin{align*}
                {}_A\sT &: \YD{A}(\cC) \rightarrow \rlYD{A^\op}{(A^\op)^\op}
(\ov{\cC}),\\
                \sT^A &: \rYD{A}(\cC) \rightarrow \rlYD{(A^\cop)^\cop}{A^\cop}
(\ov{\cC}),\\
                \sT_A &: \rYD{A}(\cC) \rightarrow \lrYD{A^\op}{(A^\op)^\op}
(\ov{\cC}).
        \end{align*}
\end{remark}

\begin{theorem}
\label{thm:side_switch_functor}
Let $A$ be a Hopf algebra in a braided category $\cC$ and 
$(X, \rho^r , \delta^r )$ a right Yetter-Drinfel'd module 
over $A$. Consider
\[
\sT(X, \rho^r , \delta^r ) = (X, \rho^- \circ (S^{-1} \otimes \id_X), (S \otimes
\id_X) \circ \delta^+ ),
\]
with $\rho^- := \rho^r \circ c_{X,A}^{-1}$ and $\delta^+ := c_{X,A} \circ
\delta^r$.
The functor
\[
\sT = (\sT_A)^A : \rYD{A}(\cC) \rightarrow \YD{A}(\cC)
\] 
has a monoidal structure $\sT_2(X,Y) : \sT(X) \otimes \sT(Y) \rightarrow \sT(X \otimes
Y)$ given by
\[
\sT_2(X,Y) := (\id_X \otimes \rho^r_Y ) \circ (\id_X \otimes c^{-1}_{Y,A}) \circ
(\delta^r_X \otimes \id_Y ).
% (\id_X \otimes \rho^-_Y ) \circ (c_{A,X} \otimes \id_Y ) \circ (\delta^+_X \otimes \id_Y ).
\] 
The monoidal functor $(\sT, \sT_2)$ is braided.
\end{theorem}

\begin{proof}
The functor $\sT : \rYD{A}(\cC) \rightarrow \YD{A}(\cC)$ is defined as the
composition of the functors in the diagram
\[
\begin{xy}
\xymatrix{
\rYD{A}(\cC ) \ar@{-->}[rr]^{\sT} \ar[d]^{\sT_A}& & \YD{A}(\cC ) \\
\lrYD{A^\op}{A^\op}(\ov{\cC} )\ar[rr]^{\sS} & & \lrYD{A^\cop}{A^\cop}(\ov{\cC} )
\ar[u]^{({}^A \sT )^{-1}}
}
\end{xy}
\]
Here $\sS$ denotes the functor of restriction along 
$S^{-1} : A^\cop \rightarrow
A^\op$ and corestriction along $S : A^\op \rightarrow A^\cop$. 
Thus, $\sT$ is a functor. Expressing the monoidal
structure in terms of braidings,
\[
        T_2(X,Y) = c_{Y,X}^{\YD{}} \circ c^{-1}_{Y,X} \,\,,
\]
and noting that the isomorphism $c_{Y,X}^{\YD{}}$ is $A$-linear 
and $A$-colinear, we see that the
morphism $T_2(X,Y)$ is $A$-(co)linear, iff
$c^{-1}_{Y,X}$ is $A$-(co)linear as a morphism $\sT X \otimes \sT Y \rightarrow
\sT(Y \otimes X)$; this is easily checked.

The inverse of $\sT_2(X,Y)$ is given by
\[
\sT_2(X,Y)^{-1} =  (\id_X \otimes \rho^r_Y ) \circ (\id_X \otimes c_{Y,A}^{-1})
\circ (\id_X \otimes S^{-1} \otimes \id_Y )
\circ (\delta_X^r \otimes \id_Y ).
\]
This follows by using that $S^{-1}$ is the antipode of $A^\op$.
We leave it to the reader to show that the equality 
\[
\sT_2(X \otimes Y, Z) \circ (\sT_2(X,Y) \otimes \id_{\sT(Z)}) = \sT_2(X,Y
\otimes Z) \circ (\id_{\sT(X)} \otimes \sT_2(Y,Z))
\] 
is a direct consequence of the Yetter-Drinfel'd condition.
We conclude that $(\sT, \sT_2)$ is a monoidal
functor.\\
Finally we show that $(\sT, \sT_2)$ is a braided monoidal functor, i.e.\
that the diagram
\[
\begin{xy}
\xymatrix{
\sT(X) \otimes \sT(Y) \ar[rr]^{\sT_2(X,Y)} \ar[d]_{c^{\YD{}}_{\sT(X),\sT(Y)}}& &
\sT(X \otimes Y) \ar[d]^{\sT \left(c^{\YD{}}_{X,Y}\right)}\\
\sT(Y) \otimes \sT(X) \ar[rr]^{\sT_2(Y,X)}& & \sT(Y \otimes X)\\
}
\end{xy}
\]
commutes. 
One easily sees by drawing the corresponding braid diagrams, that
$c^{\YD{}}_{\sT(X),\sT(Y)}$ is equal to $c_{X,Y} \circ \sT_2(X,Y)$ and $\sT
(c^{\YD{}}_{X,Y})$ is equal to $\sT_2(Y,X) \circ c_{X,Y}$. Thus we have
\begin{align*}
\sT_2(Y,X) \circ c^{\YD{}}_{\sT(X),\sT(Y)} \circ \sT_2(X,Y)^{-1} = \sT_2(Y,X)
\circ c_{X,Y} = \sT (c^{\YD{}}_{X,Y}).
\end{align*}
\end{proof}

\begin{remark}
\label{rem:side_switch_functor}
There is another braided equivalence between the same braided 
categories of left/right Yetter-Drinfel'd modules
\[
\sT' = (\sT^A)_A : \rYD{A}(\cC) \rightarrow \YD{A}(\cC) .
\] 
The functor $\sT'$ is given on objects by
\[
\sT'(X, \rho^r, \delta^r) = (X, \rho^+ \circ (S \otimes \id_X), (S^{-1} \otimes
\id_X) \circ \delta^- ).
\]
The monoidal structure $T'_2(X,Y) :\sT'(X) \otimes \sT'(Y) \rightarrow \sT'(X \otimes Y)$ on $\sT'$ is given by
\[
\sT'_2(X,Y) := (\rho^r_X \otimes \id_Y) \circ (\id_X \otimes S^{-1} \otimes
\id_Y) \circ (\id_X \otimes c^{-1}_{A,Y}) \circ (\id_X \otimes \delta^r_Y).
\]
The two monoidal functors $\sT, \sT' : \rYD{A}(\cC) \rightarrow \YD{A}(\cC)$ are
isomorphic as monoidal functors.
An isomorphism is given by the family of morphisms
\[
\theta_{\sT X} := \rho_{\sT X} \circ (S \otimes \id_X ) \circ \delta_{\sT X}.
\]
The inverse is
\[
\theta_{\sT X}^{-1} = \rho_{\sT X} \circ c_{A,X}^{-1} \circ (\id_X \otimes S^{-2}) \circ
c_{X,A}^{-1} \circ \delta_{\sT X}.
\]
A lemma for right Yetter-Drinfel'd modules that is analogous
to Lemma \ref{lem:properties_of_theta} implies that
$\theta$ is indeed a monoidal isomorphism.
\end{remark}

\subsection{Equivalence of categories from Hopf pairings}
\label{sec:HopfPairings}
In this subsection, we prove that for Hopf algebras $A$
and $B$ that are related by a non-degenerate Hopf pairing,
there is a braided monoidal equivalence
between the categories $\YD{A}(\cC)$ and $\rYD{B}(\cC)$. This equivalence is a
strict monoidal functor.
\begin{lemma}
Let $\omega : A \otimes B \rightarrow \unit$ be a non-degenerate Hopf pairing
with inverse copairing $\omega' : \unit \rightarrow B \otimes A$.
Then
\begin{align*}
{}_{\omega'}\sD^\omega : \lrYD{A^\cop}{A}(\cC) &\rightarrow
\rlYD{B^\cop}{B}(\cC) \\
(X, \rho, \delta) &\mapsto (X, (\id \otimes \omega ) \circ (\delta \otimes \id
), (\id \otimes \rho ) \circ (\omega' \otimes \id )) \,\,,
\end{align*} 
is a strict monoidal braided functor.
In particular, the two categories $\lrYD{A^\cop}{A}(\cC)$ and 
$\rlYD{B^\cop}{B}(\cC)$ are
equivalent as braided monoidal categories.

\end{lemma}
\begin{proof}
Let $(X, \rho, \delta )$ be an $A$-Yetter-Drinfel'd module. 
{}From Lemma \ref{lem:duality_functor} it is clear that $\sD
(X, \rho, \delta)$ is a $B$-module and $B$-comodule.
We have to check the Yetter-Drinfel'd condition. 
Since $X$ is an $A$-Yetter-Drinfel'd module, we have the equality
\[
%%%%%%%%%%%%%%%%%%%
\begin{grform}
\begin{scope}[scale = 0.5]
\dMult{3}{6}{1}{1}{\grau}
\dMult{2}{2}{1}{-1}{\grau}
\dAction{3}{2}{1}{1}{\grau}{black}
\dAction{3}{6}{-1}{-1}{\grau}{black}
\dPairing{3.5}{7}{2}{1}{\grau}{\scriptsize $\omega$}
\dCopairing{0.5}{1}{2}{1}{\grau}{\scriptsize $\omega'$}
\vLine{2}{2}{4}{6}{\grau}
\vLineO{4}{3}{2}{5}{black}
\vLine{4}{0}{4}{2}{black}
\vLine{2}{6}{2}{8}{black}
\vLine{0.5}{1}{0.5}{8}{\grau}
\vLine{5.5}{0}{5.5}{7}{\grau}
\end{scope}
\draw (2 , -0.3) node {$X$};
\draw (1 , 4.3) node {$X$};
\draw (2.75 , -0.3) node {$B$};
\draw (0.25 , 4.3) node {$B$};
\end{grform}
%%%%%%%%%%%%%%%%%%%
=
%%%%%%%%%%%%%%%%%%%
\begin{grform}
\begin{scope}[scale = 0.5]
\dMult{2}{3}{1}{-1}{\grau}
\dMult{5}{5}{1}{1}{\grau}
\dAction{2}{5}{1}{1}{\grau}{black}
\dAction{6}{3}{-1}{-1}{\grau}{black}
\dPairing{5.5}{6}{2}{1}{\grau}{\scriptsize $\omega$}
\dCopairing{0.5}{2}{2}{1}{\grau}{\scriptsize $\omega'$}
\vLine{5}{3}{3}{5}{black}
\vLineO{3}{3}{5}{5}{\grau}
\vLine{3}{6}{3}{8}{black}
\vLine{5}{0}{5}{2}{black}
\vLine{0.5}{2}{0.5}{8}{\grau}
\vLine{7.5}{0}{7.5}{6}{\grau}
\vLine{2}{3}{2}{5}{\grau}
\vLine{6}{3}{6}{5}{\grau}
\end{scope}
\draw (2.5 , -0.3) node {$X$};
\draw (3.75 , -0.3) node {$B$};
\draw (1.5 , 4.3) node {$X$};
\draw (0.25 , 4.3) node {$B$};
\end{grform}
.
%%%%%%%%%%%%%%%%%%%
\]
Using that $\omega$ is a Hopf pairing, $\omega'$ is a Hopf copairing
and $(\id_B \otimes \omega) \circ (\omega' \otimes \id_B) = \id_B$ we get the
equality
\[
%%%%%%%%%%%%%%%%%%%
\begin{grform}
\begin{scope}[scale = 0.5]
\dMult{5}{2}{3}{-2}{\grau}
\dMult{0}{5}{3}{2}{\grau}
\dAction{2}{1}{1}{1}{\grau}{black}
\dAction{6}{6}{-1}{-1}{\grau}{black}
\dPairing{6}{6}{2}{1}{\grau}{\scriptsize $\omega$}
\dCopairing{0}{1}{2}{1}{\grau}{\scriptsize $\omega'$}
\vLine{5}{2}{3}{5}{\grau}
\vLine{0}{1}{0}{5}{\grau}
\vLine{8}{2}{8}{6}{\grau}
\vLine{3}{0}{3}{1}{black}
\vLine{5}{6}{5}{7}{black}
\vLineO{3}{2}{5}{5}{black}
\end{scope}
\draw (1.5 , -0.3) node {$X$};
\draw (3.25 , -0.3) node {$B$};
\draw (2.5 , 3.8) node {$X$};
\draw (0.75 , 3.8) node {$B$};
\end{grform}
%%%%%%%%%%%%%%%%%%%
=
%%%%%%%%%%%%%%%%%%%
\begin{grform}
\begin{scope}[scale = 0.5]
\dMult{4}{1}{1}{-1}{\grau}
\dMult{0}{6}{1}{1}{\grau}
\dAction{4}{6}{1}{1}{\grau}{black}
\dAction{1}{1}{-1}{-1}{\grau}{black}
\dPairing{1}{1}{3}{1}{\grau}{\scriptsize $\omega$}
\dCopairing{1}{6}{3}{1}{\grau}{\scriptsize $\omega'$}
\vLine{0}{1}{5}{6}{black}
\vLineO{5}{1}{0}{6}{\grau}
\end{scope}
\draw (0 , -0.3) node {$X$};
\draw (2.25 , -0.3) node {$B$};
\draw (2.5 , 3.8) node {$X$};
\draw (0.25 , 3.8) node {$B$};
\end{grform}
%%%%%%%%%%%%%%%%%%%
\]
which is the Yetter-Drinfel'd condition for the $B$-module 
and $B$-comodule structure on $\sD (X)$.
The functor $\sD$ is strict monoidal, since the functors 
${}_{\omega'}\sD : \lMod{A^\cop}{\cC} \rightarrow \lComod{B}{\cC}$ and 
$\sD^\omega : \rComod{A}{\cC} \rightarrow \rMod{B^\cop}{\cC}$ are strict monoidal.
Finally, the braiding is preserved: 
\[
c^{\YD{}}_{\sD(X), \sD(Y)} = c^{\YD{}}_{X, Y} = \sD(c^{\YD{}}_{X, Y})\; .
\]
This follows from $(\omega \otimes \id_A) \circ (\id_A \otimes \omega') =
\id_A$.
\end{proof}

\begin{corollary}\label{cor:pairing_functor}
Let $\omega : A \otimes B \rightarrow \unit$ be a non-degenerate Hopf pairing
with inverse copairing $\omega' : \unit \rightarrow B \otimes A$.
Then
\begin{align*}
{}_{\omega'}^\omega\sD : \YD{A}(\cC) &\rightarrow \rYD{B}(\cC) \\
(X, \rho_X, \delta_X) &\mapsto (X, \rho_{\sD(X)}, \delta_{\sD(Y)})
\end{align*} 
with
\begin{align*}
\rho_{\sD(X)} &= (\id \otimes \omega ) \circ (c^{-1}_{A,X} \otimes S^{-1}) \circ
(\delta \otimes \id ) \\
\delta_{\sD(X)} &= c_{B,X} \circ (S \otimes \rho ) \circ (\omega' \otimes \id )
\end{align*}
defines a braided, strict monoidal functor.

In particular,
the categories $\YD{A}(\cC)$ and $\rYD{B}(\cC)$ are equivalent as
braided monoidal categories.
\end{corollary}

\begin{proof}
Note that $\omega : A \otimes B \rightarrow \unit$ is a Hopf pairing of the two
Hopf algebras $A^\cop$ and $B^\op$ in $\ov{\cC}$.
So we have the following composite of braided, strict monoidal functors
\[
\begin{xy}
\xymatrix{
\YD{A}(\cC) \ar@{-->}[rr] \ar[d]^{{}^A\sT}& & \rYD{(B^\op)^\cop}(\cC)
\ar[r]^-{\sS} & \rYD{B}(\cC) \\
\lrYD{(A^\cop)^\cop}{A^\cop}(\ov{\cC}) \ar[rr]^{{}_{\omega'}\sD^\omega} & &
\rlYD{B^\op}{(B^\op)^\cop}(\ov{\cC}) \ar[u]^{{}^{B^\op}\sT} &
}
\end{xy}
\]
Here $\sS$ denotes the functor of restriction along $S^{-1} : B \rightarrow
(B^\op)^\cop$ and corestriction along $S : (B^\op)^\cop \rightarrow B$. The top
line of the above diagram is the functor ${}_{\omega'}^\omega\sD$.
\end{proof}

Combining Theorem \ref{thm:side_switch_functor}, Remark
\ref{rem:side_switch_functor} and Corollary \ref{cor:pairing_functor},
we are now in a position to exhibit explicitly two 
braided equivalences
\[
\Omega, \Omega' : \YD{A}(\cC) \rightarrow \YD{B}(\cC).
\]
The first functor is the composition $\Omega := \sT \circ \sD$ with monoidal structure 
\begin{align*}
        \Omega_2(X,Y) &= \sT(\sD_2(X,Y)) \circ \sT_2(\sD X,\sD Y) = \id_{X
\otimes Y} \circ (c^{\YD{}}_{\sD Y, \sD X} \circ c^{-1}_{Y,X})\\
                &= c^{\YD{}}_{Y, X} \circ c^{-1}_{Y,X}.
\end{align*}
The second to last equal sign uses that $\sD$ is a strict braided functor.
The other functor is $\Omega' := \sT' \circ \sD$ with monoidal structure
\[
        \Omega'_2(X,Y) = \sT'(\sD_2(X,Y)) \circ \sT'_2(\sD X,\sD Y) = \left(
c_{X,Y}^{\YD{}} \right)^{-1} \circ c_{X,Y}.
\]
Graphically the functors and the monoidal structures look as follows:
\begin{align*}
        \newcommand{\rhoomega}{\raisebox{-.5\totalheight}{
        \begin{tikzpicture}
                \begin{scope}[scale=0.5]
                        \dAction{0}{2}{1}{-1}{\grau}{black}
                        \dSkewantipode{2}{1}{\grau}
                        \dSkewantipode{2}{2}{\grau}
                        \dPairing{1}{3}{1}{1}{\grau}{\tiny $\! \omega \!$}
                        \vLine{0}{2}{1}{3}{\grau}
                        \vLineO{1}{2}{0}{3}{black}
                        \vLine{1}{0}{2}{1}{\grau}
                        \vLineO{2}{0}{1}{1}{black}
                        \vLine{0}{3}{0}{4}{black}
                        \vLine{1}{-1}{1}{0}{\grau}
                        \vLine{2}{-1}{2}{0}{black}
                \end{scope}
                \draw (0.5 , -0.7) node {\tiny $B$};
                \draw (1 , -0.7) node {\tiny $X$};
                \draw (0 , 2.2) node {\tiny $X$};
        \end{tikzpicture}
        }
        }
        \newcommand{\deltaomega}{\raisebox{-.5\totalheight}{
        \begin{tikzpicture}
                \begin{scope}[scale=0.5]
                        \dAction{1}{1}{1}{1}{\grau}{black}
                        \dAntipode{0}{1}{\grau}
                        \dAntipode{0}{2}{\grau}
                        \dCopairing{0}{1}{1}{1}{\grau}{\tiny $\!\! \omega' \!\!
\vspace*{-1mm}$}
                        \vLine{1}{3}{0}{4}{black}
                        \vLineO{0}{3}{1}{4}{\grau}
                        \vLine{1}{4}{0}{5}{\grau}
                        \vLineO{0}{4}{1}{5}{black}
                        \vLine{2}{0}{2}{1}{black}
                        \vLine{2}{2}{1}{3}{black}
                \end{scope}
                \draw (1 , -0.2) node {\tiny $X$};
                \draw (0 , 2.7) node {\tiny $B$};
                \draw (0.5 , 2.7) node {\tiny $X$};
        \end{tikzpicture}
        }
        }
        \newcommand{\omegatwo}{\raisebox{-.5\totalheight}{
        \begin{tikzpicture}
                \begin{scope}[scale=0.5]
                        \dAction{0}{2}{1}{1}{\grau}{black}
                        \dAction{1}{1}{1}{-1}{\grau}{black}
                        \vLine{0}{0}{0}{1}{black}
                        \vLine{0}{1}{1}{2}{black}
                        \vLineO{1}{1}{0}{2}{\grau}
                        \vLine{2}{1}{2}{3}{black}
                \end{scope}
                \draw (0 , -0.2) node {\tiny $X$};
                \draw (1 , -0.2) node {\tiny $Y$};
                \draw (0.5 , 1.7) node {\tiny $X$};
                \draw (1 , 1.7) node {\tiny $Y$};
        \end{tikzpicture}
        }
        }
        \Omega(X,\rho_X, \delta_X) = \left( X, \rhoomega, \deltaomega \right) ,
\quad \Omega_2(X,Y) = \omegatwo ,\\
        \newcommand{\rhoomegaprime}{\raisebox{-.5\totalheight}{
        \begin{tikzpicture}
                \begin{scope}[scale=0.5]
                        \dAction{0}{2}{1}{-1}{\grau}{black}
                        \vLine{2}{1}{2}{3}{\grau}
                        \dPairing{1}{3}{1}{1}{\grau}{\tiny $\! \omega \!$}
                        \vLine{0}{2}{1}{3}{\grau}
                        \vLineO{1}{2}{0}{3}{black}
                        \vLine{2}{0}{1}{1}{black}
                        \vLineO{1}{0}{2}{1}{\grau}
                        \vLine{0}{3}{0}{4}{black}
                        \vLine{1}{-1}{1}{0}{\grau}
                        \vLine{2}{-1}{2}{0}{black}
                \end{scope}
                \draw (0.5 , -0.7) node {\tiny $B$};
                \draw (1 , -0.7) node {\tiny $X$};
                \draw (0 , 2.2) node {\tiny $X$};
        \end{tikzpicture}
        }
        }
        \newcommand{\deltaomegaprime}{\raisebox{-.5\totalheight}{
        \begin{tikzpicture}
                \begin{scope}[scale=0.5]
                        \dAction{1}{1}{1}{1}{\grau}{black}
                        \dCopairing{0}{1}{1}{1}{\grau}{\tiny $\!\! \omega' \!\!
\vspace*{-1mm}$}
                        \vLine{0}{1}{0}{5}{\grau}
                        \vLine{2}{0}{2}{1}{black}
                        \vLine{2}{2}{1}{5}{black}
                \end{scope}
                \draw (1 , -0.2) node {\tiny $X$};
                \draw (0 , 2.7) node {\tiny $B$};
                \draw (0.5 , 2.7) node {\tiny $X$};
        \end{tikzpicture}
        }
        }
        \newcommand{\omegatwoprime}{\raisebox{-.5\totalheight}{
        \begin{tikzpicture}
                \begin{scope}[scale=0.5]
                        \dAction{0}{1}{1}{-1}{\grau}{black}
                        \dAction{1}{3}{1}{1}{\grau}{black}
                        \dSkewantipode{1}{2}{\grau}
                        \vLine{2}{0}{2}{3}{black}
                        \vLine{0}{1}{1}{2}{\grau}
                        \vLineO{1}{1}{0}{2}{black}
                        \vLine{0}{2}{0}{4}{black}
                \end{scope}
                \draw (0.5 , -0.2) node {\tiny $X$};
                \draw (1 , -0.2) node {\tiny $Y$};
                \draw (0 , 2.2) node {\tiny $X$};
                \draw (1 , 2.2) node {\tiny $Y$};
        \end{tikzpicture}
        }
        }
        \Omega'(X,\rho_X, \delta_X) = \left( X, \rhoomegaprime, \deltaomegaprime
\right) , \quad \Omega'_2(X,Y) = \omegatwoprime .
\end{align*}

We summarize our findings:

\begin{theorem}
\label{thm:composedDuality}
Let $\omega : A \otimes B \rightarrow \unit$ be a non-degenerate Hopf pairing.
The categories $\YD{A}(\cC)$ and $\YD{B}(\cC)$ are 
braided equivalent via the monoidal functors
$\Omega$ and $\Omega'$ above.
\end{theorem}

We end this subsection by relating the equivalence $\Omega$  
to the equivalence $\Omega^{\rm HS}$ of rational modules over 
$\Bbbk$-Hopf algebras discussed in \cite{HS13}.

\begin{remark}
\begin{enumerate}
\item
Let $\Bbbk$ be a field and $\mathcal{L}_\Bbbk$ the category of 
linearly topologized vector spaces over $\Bbbk$. Fix a Hopf
algebra $h$ in $\mathcal{L}_\Bbbk$ and two Hopf algebras
$(R,R^\vee)$ in $\YD{h}(\mathcal{L}_\Bbbk)$ that are related
by a non-degenerate Hopf pairing. It is then shown in
\cite{HS13} that the categories $\YD{R \rtimes h}_{\rm rat}$ 
and $\YD{R^\vee \rtimes h}_{\rm rat}$ are equivalent as 
braided categories. Here,
the subscript ${\rm rat}$ denotes the subcategory of rational modules.

The non-degenerate pairing $\langle , \rangle : R^\vee \otimes R
\rightarrow \Bbbk$ and the structural morphisms of the 
bosonized Hopf algebra $R\rtimes h$ are used in
\cite[Theorem 7.1]{HS13} to construct a functor 
\[
(\Omega^{\rm HS}, \Omega^{\rm HS}_2) : \YD{R \rtimes h}_{\rm rat} \rightarrow
\YD{R^\vee \rtimes h}_{\rm rat} \,\,.
\]

In detail, the functor $\Omega^{\rm HS}$ is constructed
as follows:
Let $M$ be a rational $(R \rtimes h)$-Yetter-Drinfel'd modules and denote the left $R$-coaction by 
$\delta(m) = m_{\langle -1 \rangle} \otimes m_{\langle 0 \rangle}$.\\
The $(R^\vee\! \rtimes h)$-Yetter-Drinfel'd module $\Omega^{{\rm HS}}(M)$ is equal 
to $M$ as an $h$-Yetter-Drinfel'd module
and has the following $R^\vee$-Yetter-Drinfel'd structure
\begin{align*}
\text{action: } & &\xi m &= \langle \xi, m_{\langle -1 \rangle} \rangle m_{\langle 0\rangle} \\
\text{coaction: } & &\delta_{\Omega^{\rm HS}(M)} &= \left( c^{\YD{}}_{M,R^\vee} \circ c^{\YD{}}_{R^\vee,M} \right) (m_{[-1]} \otimes m_{[0]}),
\end{align*}
where $m_{[-1]} \otimes m_{[0]}$ is the unique element of $R^\vee \otimes M$ such that for all $r \in R$ and $m \in M$ we have
\[
rm = \left\langle m_{[-1]}, \theta_R(r)\right\rangle m_{[0]}.
\]
The monoidal structure of $\Omega^{\rm HS}$ is given by the 
family of morphisms
\begin{align*}
\Omega^{\rm HS}_2(M,N) : \Omega^{\rm HS}(M) \otimes \Omega^{\rm HS}(N) &\rightarrow \Omega^{\rm HS}(M \otimes N)\\
m \otimes n &\mapsto S_{R \rtimes h}^{-1}S_R(n_{\langle -1 \rangle})m \otimes n_{\langle 0 \rangle}.
\end{align*}
\item
In this paper, we started with a non-degenerate Hopf pairing 
$\omega: A \otimes B \rightarrow \unit$ and constructed 
an equivalence
\[
\Omega^\omega : \YD{A}(\cC) \rightarrow \YD{B}(\cC).
\]
Let $\cC$ be the category of finite dimensional Yetter-Drinfel'd modules over the finite dimensional Hopf algebra $h$.
Set $A=R$ and $B=R^\vee$ and $\omega : A \otimes B \rightarrow \Bbbk$, 
such that $\omega^-(b \otimes a) = \langle b , a \rangle$, cf. Example \ref{example:HopfPairings}. 
One can show by straight-forward computations, that our functor $\Omega^{\omega^-}$
coincides with the functor $\Omega^{\rm HS}$ on the full subcategory $\YD{R \rtimes h}_{\rm fin} \subset \YD{R \rtimes h}_{\rm rat}$
of finite dimensional $(R \rtimes h)$-Yetter-Drinfel'd modules.
\end{enumerate}
\end{remark}

\subsection{The square of \texorpdfstring{$ \Omega$}{Omega}}
\label{prop:OmegaSquare}
{}From a non-degenerate Hopf pairing $\omega : A \otimes B
\rightarrow \unit$, we obtained an equivalence
$\Omega^\omega : \YD{A}(\cC) \rightarrow \YD{B}(\cC)$. 
As noted in Example \ref{example:HopfPairings}, we also 
have a non-degenerate Hopf pairing $\omega^- : B \otimes A \rightarrow \unit$
from which we obtain an equivalence $\Omega^{\omega^-} : \YD{B}(\cC) \rightarrow
\YD{A}(\cC)$.

\begin{proposition}
The braided monoidal functor 
\[
\Omega^{\omega^-} \circ \Omega^\omega :  \YD{A}(\cC) \rightarrow \YD{A}(\cC)
\] 
is isomorphic to the identity functor.
\end{proposition}

\begin{proof}
A direct computation shows that the monoidal functors
\begin{align*}
(\Omega^\omega, \Omega^\omega_2 ) \circ \left((\Omega')^{\omega^-}, (\Omega')^{\omega^-}_2\right) \quad \text{and} \quad
\left((\Omega')^{\omega^-}, (\Omega')^{\omega^-}_2\right) \circ (\Omega^\omega, \Omega^\omega_2 )
\end{align*}
are both equal to the identity functor with identity monoidal structure.
Remark \ref{rem:side_switch_functor} implies that $(\Omega')^{\omega^-}$ is monoidally isomorphic to $\Omega^{\omega^-}$.

Alternatively, a concrete calculation shows that $\Omega^{\omega^-} \circ \Omega^\omega$ is equal to the monoidal functor that sends
the Yetter-Drinfel'd module $(X, \rho, \delta)$ to the Yetter-Drinfel'd module
\[
(X, \rho \circ (S^{-2} \otimes \id_X ) \circ c^{-1}_{A,X} \circ c^{-1}_{X,A}, c_{X,A} \circ c_{A,X} \circ (S^{2} \otimes \id_X) \circ \delta ).
\]
The monoidal structure of $\Omega^{\omega^-} \circ \Omega^\omega$ is given by the family of isomorphisms
\[
c^{\YD{}}_{Y,X} \circ c^{\YD{}}_{X,Y} \circ c^{-1}_{X,Y} \circ c^{-1}_{Y,X}.
\]
{}From the this and Lemma \ref{lem:properties_of_theta} it is clear 
that $\theta_X := \rho_X \circ (S \otimes \id_X) \circ \delta_X$ defines a monoidal isomorphism
\[
\theta : \Omega^{\omega^-} \circ \Omega^\omega \rightarrow \Id.
\]
\end{proof}

%%%%%%%%%%%%%%%%%%%%%%%%%%%%%%%%%%%%%%%%%%%%%%%%%%%%%%%%%%%%%%%%%%%%%%%%%%%%%%%%
%%%%%%%%%%%%%%%%%%%%%%%%%%%%%
%%%%%%%%%%%%%%%%%%%%%%%%%%%%%%%%%%%%%%%%%%%%%%%%%%%%%%%%%%%%%%%%%%%%%%%%%%%%%%%%
%%%%%%%%%%%%%%%%%%%%%%%%%%%%%
%%%%%%%%%%%%%%%%%%%%%%%%%%%%%%%%%%%%%%%%%%%%%%%%%%%%%%%%%%%%%%%%%%%%%%%%%%%%%%%%
%%%%%%%%%%%%%%%%%%%%%%%%%%%%%  
%%%%%%%%%%%%%%%%%%%%%%%%%%%%%%%%%%%%%%%%%%%%%%%%%%%%%%%%%%%%%%%%%%%%%%%%%%%%%%%%
%%%%%%%%%%%%%%%%%%%%%%%%%%%%%
  
\section{Partial dualization of a Hopf algebra}\label{sec:PartialDualization}

We now present the main construction of this article: Let $H$ be a Hopf algebra
in a braided category $\cC$, $A$ be a Hopf subalgebra
and $\pi:H\to A$ a Hopf algebra projection. Moreover, let
$B$ be a Hopf algebra in $\cC$ and $\omega:A\otimes B\to\unit$ a non-degenerate
Hopf pairing. 
These data constitute a \emph{partial dualization
datum} $\cA$ for the Hopf algebra $H$ to which we associate 
a \emph{partial dualization} $r_\cA(H)$,
a new Hopf algebra in the braided category $\cC$. The construction
makes use of the functors introduced in Section 3 that
relate various categories of Yetter-Drinfel'd modules.
We show  that the assignment $H\mapsto r_\cA(H)$ is 
involutive up
to an isomorphism. We also prove a fundamental equivalence of
braided categories
$$\YD{H}(\cC)\cong \YD{r_\cA(H)}(\cC).$$
This equivalence implies that the Drinfel'd doubles of $H$ and
$r_\cA(H)$ are Morita-equivalent Hopf algebras.

\subsection{Main construction}
We start with some definitions:

\begin{definition}\label{d:PartialDualization}
Let $\cC$ be a braided monoidal category.
A {\em partial dualization datum} $\cA=(H\stackrel{\pi}{\to}A,B,\omega)$
for a Hopf algebra $H$ in $\cC$ consists of
    \begin{itemize}
      \item a Hopf algebra projection $\pi:H\to A$ to a Hopf subalgebra
            $A\subset H$,
      \item a Hopf algebra $B$ with a non-degenerate Hopf pairing 
            $\omega: A\otimes B\rightarrow \unit_\cC$. 
    \end{itemize}   
\end{definition}
Given a partial dualization datum $\cA$ for a Hopf algebra $H$
in $\cC$, the {\em partial dualization} $r_\cA(H)$ is the
following Hopf algebra in $\cC$:
\begin{itemize}
\item 
By the Radford projection theorem \nref{thm:RadfordProjection},
the projection $\pi:H\to A$ induces a Radford biproduct 
decomposition of $H$
$$H\cong K\rtimes A \,\,, $$
where $K:=H^{\text{coin}(\pi)}$ is a Hopf algebra in the braided
category $\YD{A}(\cC)$.
\item 
The non-degenerate Hopf pairing $\omega:A\otimes B\to \unit$ 
induces by Theorem \nref{thm:composedDuality} a braided 
equivalence:
$$\Omega:\; \YD{A}\left(\cC\right)\stackrel\sim\to\YD{B}\left(\cC\right).$$
Thus, the image of the Hopf algebra $K$ in $\YD{A}(\cC)$ 
under the braided functor $\Omega$ is a Hopf algebra 
$L:=\Omega\left(K\right)$ in 
the braided category $\YD{B}\left(\cC\right)$.
\item 
The Radford biproduct from Definition
\nref{d:RadfordBiproduct} of $L$ over $B$ allows us
to introduce the partially
dualized Hopf algebra,
$$r_\cA(H):=L\rtimes B \,\, ,$$
which is a Hopf algebra in $\cC$.
As a Radford biproduct, it comes 
with a projection $\pi':r_\cA(H)\to B$. 
\end{itemize}

We summarize:

\begin{definition}
For a partial dualization datum 
$\cA = (H \xrightarrow{\pi} A, B, \omega)$, we call 
the Hopf algebra $r_\cA(H)$ in $\cC$
the \emph{partial dual} of $H$ with respect to $\cA$. 
\end{definition}

Our construction is inspired by the calculations in
\cite{HS13} using smash-products. In Section
\nref{sec:ExampleNicholsAlgebra}, we explain the relation of
these calculations to our general construction.

\subsection{Involutiveness of partial dualizations}
The Hopf algebra $r_\cA(H)$ comes with a projection
to the subalgebra $B$. The two Hopf pairings 
$\omega^\pm: B\otimes A\to \unit_\cC$ from
Example  \nref{example:HopfPairings} (2)
yield two possible partial dualization data for $r_\cA(H)$:
\begin{align*}
  \cA^+&=(r_\cA(H)\stackrel{\pi'}{\to}B,A,\omega^+)\\
  \cA^-&=(r_\cA(H)\stackrel{\pi'}{\to}B,A,\omega^-) \,\, .
\end{align*}
Recall from Subsection \nref{prop:OmegaSquare} the natural 
isomorphism 
  $$\theta:\;\Omega^{\omega^-} \circ \Omega^{\omega}\cong \Id_{\YD{A}(\cC)}
  \,\,.
$$  
In a similar way, one has a natural isomorphism
$$\tilde\theta:\; \Omega^{\omega} \circ \Omega^{\omega^+}\cong \Id_{\YD{B}(\cC)}.$$
\begin{corollary}\label{corr:involutive}
The two-fold partial dualization 
$r_{\cA^-}(r_\cA(H))$ is isomorphic to $H$, as  Hopf algebra in the
braided category $\cC$. A
non-trivial isomorphism of Hopf algebras is 
    $$r_{\cA^-}(r_\cA(H))=\Omega^{\omega^-}(\Omega^{\omega}(K))\rtimes A
    \xrightarrow{\;\theta_K\otimes id_A\;}K\rtimes A= H,$$
with $\theta_K=\rho_K\circ(S_A\otimes id_K)\circ\delta_K$ as 
    in Lemma \nref{lem:properties_of_theta}.
\end{corollary}

%\section{The Equivalence of Categories
%\texorpdfstring{$\YD{R}{\cC}\cong\YD{r_\cA(R)}{\cC}$}{YD(R)=YD(r_\cA(R))}}
\subsection{Relations between the representation categories}
\label{sec:YetterDrinfeldsCoincide}

It is natural to look for relations between categories of 
representations of a Hopf algebra $H$ in $\cC$ and its
partial dualization $r_\cA(H)$:

\begin{theorem}\label{thm:YetterDrinfeldsCoincide}
  Let $H$ be a Hopf algebra in a braided category $\cC$, let
  $\cA=(H\stackrel{\pi}{\to}A,B,\omega)$ be a partial dualization datum and
  $r_\cA(H)$ the partially dualized Hopf algebra. Then
  the equivalence of braided categories
  $$\Omega:\;\YD{A}(\cC )\to \YD{B}(\cC)$$
 from Theorem \nref{thm:composedDuality} induces an {\em equivalence of braided categories}:  
  $$\YD{H}\left(\cC\right)
  \cong\YD{K}\left(\YD{A}\left(\cC\right)\right)
  \xrightarrow{\quad\tilde{\Omega}\quad}
    \YD{L}\left(\YD{B}\left(\cC\right)\right)
  \cong\YD{r_\cA(H)}\left(\cC\right).$$
\end{theorem}
\begin{proof}
The Hopf algebra $L\in\YD{B}(\cC)$ was defined as the image of
$K\in\YD{A}(\cC)$ under the
functor $\Omega$, i.e.\
$L=\Omega(K)$. The braided equivalence $\Omega$ induces an
equivalence $\tilde\Omega$ of Yetter-Drinfel'd modules over the Hopf algebra 
$K$ in the braided category $\YD{A}(\cC)$ to 
Yetter-Drinfel'd modules over the Hopf algebra 
$L=\Omega(K)$ in $\YD{B}(\cC)$
  \begin{align*}
  \YD{K}\left(\YD{A}\left(\cC\right)\right)
  \xrightarrow{\quad\tilde{\Omega}\quad}
    &\YD{\Omega\left(K\right)}\left(\YD{B} \left(\cC\right)\right)\\
  &=:\YD{L}\left(\YD{B} \left(\cC\right)\right) .
\end{align*}
By Theorem \nref{thm:RadfordProjection}, the source category
of $\tilde\Omega$ is 
$$\YD{K}\left(\YD{A}\left(\cC\right)\right)
\cong\YD{K\rtimes A}\left(\cC\right)
=\YD{H}\left(\cC\right)\,\,. $$
Similarly, we have for the target category of $\tilde\Omega$
$$ \YD{L}\left(\YD{B}\left(\cC\right)\right)
\cong\YD{L\rtimes B}\left(\cC\right)
=\YD{r_\cA(H)}\left(\cC\right) \,\, . $$
Altogether, we obtain a braided equivalence
$$\YD{H}\left(\cC\right)
  \cong\YD{K}\left(\YD{A}\left(\cC\right)\right)
    \xrightarrow{\;\tilde{\Omega}\;}
    \YD{L}\left(\YD{B}\left(\cC\right)\right)
  \cong\YD{r_\cA(H)}\left(\cC\right).$$
\end{proof}

If $\cC$ is the category of vector spaces over a field
$\Bbbk$, Yetter-Drinfeld modules over a Hopf algebra
$H$ can be described as modules over the Drinfel'd double
$\cD(H)$. For Hopf algebra $H$ in a general braided category $\cC$, 
the appropriate notion of a Drinfel'\(f\)d double $\cD(H)$
has been introduced in \cite{BV12} such that 
a braided equivalence
$\lMod{\cD(H)}{\cC} \cong \YD{H}\left(\cC\right)$ holds.
Hence Theorem \ref{thm:YetterDrinfeldsCoincide} implies

\begin{corollary}\label{cor:DoublesCoincide}
The categories of left modules over the Drinfeld double $\cD(H)$ of a 
Hopf algebra $H$ and over the Drinfel'd double $\cD(r_\cA(H))$ of
its partial dualization $r_\cA(H)$ 
are braided equivalent. 
\end{corollary}

\section{Examples}\label{sec:Examples}

We illustrate our general construction in three different
cases:

\subsection{The complex group algebra of a semi-direct product}\label{sec:ExampleGroupring}
For the complex Hopf algebra associated to a finite group $G$,
we take
  $$\cC=\text{vect}_\C \qquad H=\C[G].$$
To get a partial dualization datum for $H$, suppose that there
is a split extension $N\to G\to Q$, which allows
us to identify $Q$ with a subgroup of $G$, i.e.\
$G=N\rtimes Q$. 
We then get a split Hopf algebra projection to
  $A:=\C[Q]$:
  $$\pi:\;\C[G]\to\C[Q].$$
The coinvariants of $H$ with respect to $\pi$, which
by Theorem \nref{thm:RadfordProjection} have the structure of
a Hopf algebra $K\in \YD{A}(\cC)$, turn out to be
  $$K:=H^{\text{coin}(\pi)}=\C[N]\,\,.$$ 
The $A$-coaction on the $A$-Yetter-Drinfel'd module $K$ is trivial,
since the Hopf algebra $H$ is cocommutative. 
The $A$-action on $K$ is non-trivial; it is given by the 
action of $Q\subset G$ on the normal subgroup $N$. 
Because of the trivial $A$-coaction, the 
self-braiding of $K$ in $\YD{A}$ is trivial; thus $K$ is even 
a complex Hopf algebra. Writing $H$ as in 
Theorem \nref{thm:RadfordProjection} as a Radford biproduct,
we recover 
$$H = K\rtimes A = \C[N] \rtimes \C[Q].$$
Since the $A$-coaction on $K$ is trivial, the
coalgebra structure is just given by the tensor product
of the coalgebra structures on the group algebras.

As the dual of $A$, we take the commutative Hopf algebra of
functions on $Q$, $B:=\C^Q$; 
we denote its canonical basis
by $(e_q)_{q\in Q}$; the Hopf pairing $\omega$ 
is the canonical evaluation. This gives the partial 
dualization datum 
$$\cA=(\C[G]\xrightarrow{\pi}\C[Q],\C^Q,\omega) \,\,.$$
Since the coaction of $A$ on $K$ is trivial, the morphism 
$\Omega_2(K,K)$ from the monoidal structure on $\Omega$
is trivial. Hence the functor 
  $\Omega^\omega$ maps $K$ to the same complex Hopf
algebra  
  $$L:=\Omega^\omega(K)\cong \C[N]\,\, ,$$
which however has now to be seen as  a Yetter-Drinfel'd module
over $\C^Q$, i.e.\ $L\in  \YD{\C^Q}$: $L$ has trivial action 
of $B=\C^Q$ and the
coaction is given by the dualized action of $Q$ on $N$ 
  $$n\longmapsto \sum_{q\in Q} e_q\otimes q^{-1}nq.$$
The partial dualization $r_\cA(H)$ is, by definition, 
the Radford biproduct 
  $$r_\cA(H)=L\rtimes B=\C[N]\rtimes \C^Q.$$
In this biproduct, the algebra structure is given by the 
tensor product of algebras.

An $H$-module is a complex $G$-representation.
To give an
alternative description of the category $\lMod{r_\cA(H)}{}$,
we make the definition of $r_\cA(H)$-modules explicit:
An $r_\cA(H)$-module $V$, with $r_\cA(H)=\C[N]\rtimes \C^Q$
has the structure of a $\C^Q$-module and thus of a
$Q$-graded vector space: $V=\bigoplus_{q\in Q} V_q$.
Moreover, it comes with an action of $N$ denoted by
$n.v$ for $n\in N$ and $v\in V$. Since the algebra structure
is given by the tensor product of algebras, the $N$-action preserves
the $Q$-grading.
The tensor product of two $r_\cA(H)$-modules $V$ and $W$
is graded in the obvious way,
$$(V\otimes W)_q=\bigoplus_{q_1q_2=q} V_{q_1}\otimes W_{q_2}.$$
The non-trivial comultiplication 
      $$\Delta_{\C[N]\rtimes \C^Q}(n)=\sum_{q\in Q}(n\otimes
      e_q)\otimes(q^{-1}nq\otimes 1)$$
for the Radford biproduct 
implies a non-trivial $N$-action on the tensor product:
on homogeneous components $V_{q_1}$ and $W_{q_2}$, with
$q_1,q_2\in Q$, we have for $n\in N$
$$n.(V_{q_1}\otimes W_{q_2})=(n.V_{q_1})\otimes((q_1^{-1}nq_1).V_{q_2}).$$
We are now in a position to give the alternative description
of the category $\lMod{r_\cA(H)}{}$.
We denote by ${\rm vect}_G$ the monoidal category of $G$-graded
finite-dimensional
complex vector spaces, with the monoidal structure inherited from the
category of vector spaces. Representatives of the isomorphism
classes of simple objects are given by the one-dimensional
vector spaces $\C_g$ in degree $g\in G$.
Given a subgroup $N\leq G$, the object
$\C[N]:=\oplus_{n\in N}\C_n$ has a natural structure
of an associative, unital algebra in ${\rm vect}_G$.
It is thus possible to consider $\C[N]$-bimodules in the
monoidal category ${\rm vect}_G$; together with the tensor
product $\otimes_{\C[N]}$, these bimodules form a monoidal
category $\Bimod{\C[N]}{{\rm vect}_G}$.
In this setting, we have the following description of
the category $\lMod{r_\cA(H)}{}$:

  \begin{lemma}
      The monoidal category $\lMod{r_\cA(H)}{}$  is monoidally
      equivalent to the category of $\Bimod{\C[N]}{{\rm vect}_G}$.
  \end{lemma}
The braided equivalence of Yetter-Drinfel'd modules
over $H$ and $r_\cA(H)$ established in Theorem
\ref{thm:YetterDrinfeldsCoincide}, more precisely
the braided equivalence of the categories of modules over their
Drinfel'd doubles from  Corollary  \nref{cor:DoublesCoincide},
implies the braided equivalence   
  $$\cZ(\lMod{\mathbb{C}[G]}{})\cong \cZ(\text{vect}_G)\cong 
    \cZ(\Bimod{\C[N]}{\text{vect}_G})\,\, $$
which has been shown in \cite[Theorem 3.3]{Schau01} in a
more general context.

  \begin{proof}
It suffices to specify a monoidal functor 
 $$\Phi: \Bimod{\C[N]}{\text{vect}_G}\to
 \lMod{r_\cA(H)}{}$$
that is bijective on the spaces of morphisms and to give a preimage for
every  object $D\in \lMod{r_\cA(H)}{}$. 
Suppose that $B$ is a $\C[N]$-bimodule in the category
$\text{vect}_G$, i.e.\ $B=\bigoplus_{g\in G} B_g$,
with $\C[N]$-actions denoted by arrows $\rightharpoonup,\leftharpoonup$.

To define the functor $\Phi$ on objects,
consider for a bimodule $B$ the
$Q$-graded vector space $\Phi(B):=\oplus_{q\in Q}B_q\subset B$,
obtained by retaining only the homogeneous components with degree in
$Q\subset G$. A left $N$-action is defined for any homogeneous 
vector $v_q\in\Phi(B)_q$ by
$$n.v_q:= n\rightharpoonup v_q \leftharpoonup (q^{-1}n^{-1}q).$$
Moreover,
$$n.v_q=n\rightharpoonup v_q \leftharpoonup (q^{-1}n^{-1}q)\;\in
\Phi(B)_{nq(q^{-1}n^{-1}q)}=\Phi(B)_q\,\,,$$
since $\rightharpoonup,\leftharpoonup$ are morphisms in
$\text{vect}_G$. Thus the $N$-action preserves the $Q$-grading;
we conclude that $\Phi(B)$ is an object in $\lMod{r_\cA(H)}{}$.
      
On the morphism spaces, the functor $\Phi$ acts by restriction 
to the vector subspace $\Phi(B)\subset B$. 
We show that this gives a bijection on morphisms: Suppose $\Phi(f)=0$, then
$f(v_q)=0$ for all $v_q$ with grade $q\in Q$. For an arbitrary $v_g\in
B$ with grade $g\in G$, we may write $g=nq$ with $n\cdot q\in N\rtimes
Q$ and get an element $n^{-1}\rightharpoonup v_g$ of 
degree $q$. Using that 
$f$ is a morphism of $\C[N]$-bimodules, we find  
$f(v_g)=n\rightharpoonup f(n^{-1}\rightharpoonup v_g)= 0$. 
Thus $\Phi$ is injective on
morphisms. To show surjectivity, we take
a morphism
$f_\Phi: \Phi(B)\rightarrow \Phi(C)$;
writing again $g=nq$, we define a linear map
$f:B\rightarrow C$ on $v_g\in V_g$ by $f(v_g):=n\rightharpoonup f_\Phi(n^{-1}\rightharpoonup v_g)$. This linear map is,  by construction, a morphism 
of left $\C[N]$-modules in $\text{vect}_G$. It remains to verify
that $f$  is also a morphism of right
$\C[N]$-modules. We note that for $g\in G$, the decomposition 
$g=nq$ with $n\in N$ and $q\in Q$ implies 
$gm=(nqmq^{-1})q$ with $nqmq^{-1}\in N$ for all $m\in N$.
We thus find:

\begin{align*}
f(v_g\leftharpoonup m)
 &=(nqmq^{-1})\rightharpoonup f_\Phi((nqmq^{-1})^{-1}\rightharpoonup v_g
\leftharpoonup m)\\
 &=(nqmq^{-1})\rightharpoonup
f_\Phi((qm^{-1}q^{-1})\rightharpoonup(n^{-1}\rightharpoonup v_g)
\leftharpoonup m)\\
 &=(nqmq^{-1})\rightharpoonup
f_\Phi((qm^{-1}q^{-1}).(n^{-1}\rightharpoonup v_g))\\
&=(nqmq^{-1})\rightharpoonup
(qm^{-1}q^{-1}).f_\Phi((n^{-1}\rightharpoonup v_g))\\
&=(nqmq^{-1})\rightharpoonup (qm^{-1}q^{-1})\rightharpoonup
f_\Phi(n^{-1}\rightharpoonup v_g)\leftharpoonup m\\
&=n\rightharpoonup f_\Phi(n^{-1}\rightharpoonup v_g)\leftharpoonup m\\
&=f(v_g)\leftharpoonup m
\end{align*}
In the forth identity, we used 
that $f_\Phi$ is $r_\cA(H)$-linear.

Next we show that $\Phi$ has a natural structure of a monoidal functor. Recall that the
      tensor product $V\otimes W$ in $\text{vect}_G$ (resp. $\text{vect}_Q$) is
      defined as the tensor product of vector spaces with diagonal grading
      $V_g\otimes W_h\subset (V\otimes W)_{gh}$. Furthermore, the tensor
      product in $\Bimod{\C[N]}{}$ is defined by $\otimes_{\C[N]}$. On the other
      side, the tensor product $\otimes$ in $\lMod{r_\cA(H)}{}$ is the tensor product
      of modules over the Hopf algebra $r_\cA(H)=\C[N]\rtimes \C^Q$ with
      diagonal grading and action
      \[n.(\Phi(B)_{q_1} \otimes \Phi(B)_{q_2})=(n.\Phi(B)_{q_1})\otimes((qnq^{-1}).\Phi(B)_{q_2}).\]
      We now show that the canonical projection of vector spaces
      $B\otimes C\to B\otimes_{\C[N]}C$ gives rise to a monoidal     
      structure on $\Phi$:
      $$\Phi_2:\Phi(B)\otimes \Phi(C)\rightarrow \Phi(B\otimes_{\C[N]}C).$$
      It is clear that this map is compatible with the $Q$-grading. The
      compatibility with the $N$-action is calculated as follows: for $n\in
      N,\;b\in B_{q_1}\;c\in C_{q_2}$:
      \begin{align*}
        \left(n.(b\otimes c)\right)
        &=n.b\otimes (q_1^{-1} n q_1).c\\
        &\stackrel{\phi}{\longmapsto} (n.b)\otimes_{\C[N]}(q_1^{-1} n q_1).c\\
        &=(n\rightharpoonup b \leftharpoonup (q_1^{-1}n^{-1}q_1))
          \otimes_{\C[N]} ((q_1^{-1} n q_1)\rightharpoonup 
          c \leftharpoonup (q_2^{-1}q_1^{-1} n q_1q_2))\\
        &=(n\rightharpoonup b \leftharpoonup (q_1^{-1}n^{-1}q_1q_1^{-1} n q_1))
          \otimes_{\C[N]} (c \leftharpoonup (q_2^{-1}q_1^{-1} n q_1q_2))\\
        &=(n\rightharpoonup b)
          \otimes_{\C[N]} (c \leftharpoonup ((q_1q_2)^{-1} n (q_1q_2))\\
        &=n.\left(b\otimes_{\C[N]}c\right).
      \end{align*}
Moreover, $\Phi_2$ is clearly compatible with
the associativity constraint.
We now
show $\Phi_2$ is bijective by giving an explicit inverse: Consider an element
$v\otimes w\in B\otimes_{\C[N]}C$ which is in 
$\Phi(B\otimes_{\C[N]}C)\subset B\otimes_{\C[N]}C$.
Restricting to homogeneous elements, we take 
$v\otimes w\in(B\otimes_{\C[N]}C)$ with
$v$ of degree $g\in G$ and $w$ of degree $h\in G$.
Since $v\otimes w$ is even in the subspace $\Phi(B\otimes_{\C[N]}C)$,
we have $q:=gh\in Q$. Writing $h=n'q'$ with $n'\in N$ and
$q'\in Q$, have in the tensor product over $\C[N]$ the identity
$v\otimes
w=(v\leftharpoonup n')\otimes ((n')^{-1}\rightharpoonup w)$ with tensor
factors both graded in $Q$, hence in $\Phi(B)\otimes\Phi(C)$. We may now
define the inverse $\Phi_2^{-1}(v\otimes w):=(v\leftharpoonup n)\otimes
(n^{-1}\rightharpoonup w)$, which is a left- and right-inverse of $\Phi_2$.
Finally the monoidal units in the categories are $\C_1,\;1\in Q$ resp.
$\C[N]$; then $N\cap Q=\{1\}$ implies that there is an obvious  
isomorphism $\C_1\cong \Phi(\C[N])$. Hence
$\Phi$ is a monoidal functor.

To verify that $\Phi$ indeed defines an equivalence of tensor
categories, it remains to construct for 
each $V\in \lMod{r_\cA(H)}{}$ an object     
$D\in\Bimod{\C[N]}{\text{vect}_G}$ such that $\Phi(D)\cong V$.

The following construction could be understood as an induced
corepresentation via the cotensor product, but we prefer to keep the
calculation explicit: For $V=\bigoplus V_q$ consider the vector space
      $$D:=\bigoplus_{q\in Q} \C[N]\otimes V_q.$$
Since $G=NQ$, the vector space $D$ is naturally endowed with a $G$-grading.
Left multiplication on $\C[N]$ gives a natural
left $N$-action $\rightharpoonup$
via left-multi\-pli\-cation on $\C[N]$, which is clearly a morphism in
$\text{vect}_G$. We define a right $N$-action on
$D$ by
$$(n\otimes v_q)\leftharpoonup m\;:=\;
      n(qmq^{-1})\otimes (qn^{-1}q^{-1}).v_q.$$
Since the left action preserves the $Q$-grading, 
the vector $(n\otimes v_q)\leftharpoonup m$ has degree
$n(qmq^{-1})q=(nq)m$; thus also the right action 
$\leftharpoonup$ is a morphism in $\text{vect}_G$.\\

We finally verify that $\Phi(D)\cong V$: the homogeneous components
of $D$ with degree in the subgroup $Q$ only
are spanned by elements $1\otimes v_q$, hence we can 
identify $\Phi(D)$ with $V$. We check that the $N$-action 
defined on $\Phi(D)$ coincides with the one on $V$ we started with:
      \begin{align*}
        n.(1\otimes v_q)
        &=n\rightharpoonup (1\otimes v_q) \leftharpoonup (q^{-1}n^{-1}q)\\
        &=n(q(q^{-1}n^{-1}q)q^{-1})\otimes (q(q^{-1}n^{-1}q)^{-1}q^{-1}).v_q\\
        &=1\otimes n.v_q.
      \end{align*}
\end{proof}

\subsection{The Taft algebra}\label{sec:ExampleTaftAlgebra}

Fix a natural number $d$ and
let $\zeta \in \C$ be a primitive $d$-th root of unity. 
We consider the Taft algebra $T_\zeta$ which is a complex
Hopf algebra. As an algebra, $T_\zeta$ is generated by 
two elements $g$ and $x$ modulo the relations
\[
g^d=1, \,\, x^d=0 \quad \text{and} \quad gx = \zeta xg.
\]
A coassociative comultiplication on $T_\zeta$ is defined
by the unique algebra homomorphism 
$\Delta : T_\zeta \to T_\zeta \otimes T_\zeta$ with
\[
\Delta(g) = g \otimes g \quad \text{and} \quad \Delta(x) = g \otimes x + x \otimes 1.
\]
\begin{lemma}
Let $\zeta$ and $\xi$ be primitive $d$-th roots of unity.
If there exists an isomorphism $\psi : T_{\zeta} \rightarrow T_{\xi}$ 
of Hopf algebras, then $\zeta = \xi$.
\end{lemma}
\begin{proof}
The set $\{ x^ng^m \mid 0 \leq n,m <d \}$ is a $\C$-basis of 
$T_{\xi}$ consisting of eigenvectors for the automorphisms 
\begin{align*}
\ad_h :T_{\xi} &\to T_{\xi} , \quad
a \mapsto hah^{-1},
\end{align*}
with $h = g^c$ for $c \in \{1,2, \ldots, N-1\}$.

Suppose that $\psi : T_{\zeta} \to T_{\xi}$ is a Hopf algebra 
isomorphism. 
Then the image $h := \psi(g)$ of the generator $g$ of $T_\zeta$
is equal to $g^c \in T_{\xi}$ for some $c \in \{1,2, \ldots, d-1\}$. 
The generator $x$ of $T_\zeta$  is mapped by the algebra
homomorphism $\psi$  to an eigenvector $y:=\psi(x)$
of $\ad_h$ to the eigenvalue $\zeta$:
\begin{align*}
hyh^{-1}=\psi(gxg^{-1})=\zeta y.
\end{align*}
Since $\xi$ is a primitive root of unit, we find $0<n<d$ such that
$\zeta = \xi^n$.
Thus $y$ is an element of the $\C$-linear subspace $\langle x^ng^m \mid 0 \leq m < d\rangle_\C$ of $T_\xi$.
This implies that $y^k = 0$ for $k$ the smallest number such that $kn \geq d$. Since $\psi$ is an isomorphism, $n$ has to be $1$ and hence $\zeta = \xi$.
\end{proof}

Denote by $A$ the Hopf subalgebra of $T_\zeta$ generated by $g$.
We will deduce from Proposition \ref{prop:5.5} that the partial dual of $T_\zeta$ with respect to $A$ is isomorphic to $T_\zeta$. Hence we 
have to enlarge the class of complex Hopf algebras beyond
Taft algebras to get a non-trivial example.

Let $N$ be a natural number and $d$ a divisor of $N$.
Now let $\zeta$ be a primitive $d$-th root of unity and $q$ a primitive $N$-th root of unity. 
Let $c+N\Z$ be the unique residue class such that $\zeta = q^c$.
Define $\hat{T}_{\zeta,q}$ as the $\C$-algebra
\[
\hat{T}_{\zeta,q} := \langle x,g \mid g^N=1, x^d=0, gx = \zeta xg\rangle
\]
and define $\check{T}_{\zeta,q}$ as the $\C$-algebra
\[
\check{T}_{\zeta,q} := \langle x,g \mid g^N = 1, x^d = 0, gx = qxg\rangle.
\]
Both algebras are finite-dimensional of dimension $Nd$.

One checks the following
\begin{lemma}
Let $\hat{T}_{\zeta,q}$ and $\check{T}_{\zeta,q}$ be the algebras from above.
The unique algebra homomorphisms 
$\hat\Delta: \hat{T}_{\zeta,q}\to \hat{T}_{\zeta,q}\otimes \hat{T}_{\zeta,q}$
and $\check\Delta: \check{T}_{\zeta,q}\to \check{T}_{\zeta,q}\otimes 
\check{T}_{\zeta,q}$ defined on the generators by
\begin{align*}
\hat{\Delta}(g) & := g \otimes g & \hat{\Delta}(x) & := g \otimes x + x \otimes 1 \\
\check{\Delta}(g) & := g \otimes g & \check{\Delta}(x) & := g^c \otimes x + x \otimes 1 
\end{align*}
give the structure of an coassociative counital Hopf algebra on
$\hat{T}_{\zeta,q}$ and $\check{T}_{\zeta,q}$, respectively.

Furthermore, we have exact sequences of Hopf algebras, with $k := \frac{N}{d}$
\begin{align*}
\xymatrix{
\C[\Z_k] \ar[r] & \hat{T}_{\zeta,q} \ar[r] & T_\zeta
}
\\
\xymatrix{
T_\zeta \ar[r] & \check{T}_{\zeta,q} \ar[r] & \C[\Z_k]
}
.
\end{align*}
\end{lemma}

The Hopf subalgebra $A \subset \hat{T}_{\zeta,q}$ 
generated by the grouplike element $g$ and the Hopf subalgebra 
$B \subset \check{T}_{\zeta,q}$ generated by $g$ are both isomorphic to the 
complex group Hopf algebra $\C[\Z_N]$. To apply a partial
dualization, we need a Hopf pairing; it is given by the following
lemma whose proof we leave to the reader:

\begin{lemma}\label{5.4}
Let $q$ be an $N$-th primitive root of unity and let $g \in \C[\Z_N]$ be a generator of the cyclic group $\Z_N$.
\begin{enumerate}
\item
The bilinear form $\omega : \C[\Z_N] \times \C[\Z_N] \to \Bbbk$ given by $\omega(g^n, g^m) = q^{nm}$ is a Hopf pairing.
\item
The linear map $\omega' : \Bbbk \rightarrow \C[\Z_N] \otimes \C[\Z_N]$ 
with 
\[
\omega'(1_\Bbbk)=
\frac{1}{N} \sum_{k,\ell=1}^{N} q^{-k\ell} g^k \otimes g^\ell
\]
is the inverse copairing of $\omega$.
\end{enumerate}
\end{lemma}

The partial dual of $\hat{T}_{\zeta,q}$ with 
respect to $A$ and $\omega$ is isomorphic to $\check{T}_{\zeta,q}$:

\begin{proposition}\label{prop:5.5}
Let $N$ be a natural number and $d$ be a divisor of $N$.
Let $\zeta$ be a primitive $d$-th root of unity and $q$ a primitive $N$-th root of unity with $q^c=\zeta$.
Let $A \subset \hat{T}_{\zeta,q}$ and $B \subset \check{T}_{\zeta,q}$ be as above 
and $\omega : A \otimes B \rightarrow \Bbbk$ the non-degenerate Hopf pairing from Lemma \ref{5.4}.
\begin{enumerate}
\item
The algebra homomorphism $\pi : \hat{T}_{\zeta,q} \rightarrow A$ which sends $g$ to $g$ and $x$ to $0$ is a Hopf algebra projection.
\item
The partial dualization of $\hat{T}_{\zeta,q}$ with respect to the partial dualization datum $(\hat{T}_{\zeta,q}\xrightarrow{\pi} A, B, \omega)$ 
is isomorphic to $\check{T}_{\zeta,q}$.
\end{enumerate}
\end{proposition}
In particular, for $N=d$, we have $\hat T_{\zeta,q}=\check T_{\zeta,q}$.

\begin{proof}
The space of coinvariants $K := \hat{T}_{\zeta,q}^{\;{\rm coin}(\pi)} = \{a \in \hat{T}_{\zeta,q} \mid \hat{\Delta}(a) = a \otimes 1\}$ 
equals the $\C$-linear span of $\{1,x,x^2,\ldots,x^{d-1}\}$.
Remark \ref{rem:projection_thm} implies that $K$ is a Yetter-Drinfel'd module with
$A$-action $\rho : A \otimes K \rightarrow K$ and $A$-coaction $\delta : K \rightarrow A \otimes K$ given by
\begin{align*}
\rho &: g \otimes x \mapsto gxg^{-1} = \zeta x,\\
\delta &: x \mapsto \pi(g) \otimes x = g \otimes x = x_{(-1)} \otimes x_{(0)}.
\end{align*} 
Moreover, $K$ has the structure of a Hopf algebra in $\YD{A}$
with multiplication and comultiplication given by
\begin{align*}
\mu &: x \otimes x \mapsto x^2,\\
\Delta &: x \mapsto 1 \otimes x + x \otimes 1\,\,.
\end{align*}
The dualization functor $(\Omega, \Omega_2)$ from Section
\ref{sec:HopfPairings} for the
Hopf pairing $\omega : A \otimes A \rightarrow \Bbbk$
yields the 
$B$-Yetter-Drinfel'd module $L = \langle 1,x,x^2,\ldots, x^{d-1}\rangle_\C$ with action $\rho' : B \otimes L \rightarrow L$ and coaction $\delta' : L \rightarrow B \otimes L$ given by
\begin{align*}
\rho' &: g \otimes x \mapsto \omega(x_{(-1)},g)x_{(0)} = qx,\\
\delta' &: x \mapsto \frac{1}{N} \sum_{k,\ell = 1}^N q^{-k \cdot \ell} g^k \otimes \rho(g^\ell \otimes x) =  \frac{1}{N} \sum_{k,\ell = 1}^N (q^{-k}q^c)^{\ell} g^k \otimes x = g^c \otimes x.
\end{align*}
The Yetter-Drinfel'd module $L$ has a natural structure of a Hopf 
algebra in $\YD{B}$ with multiplication $\mu' = \mu \circ \Omega_2(K,K)$ and comultiplication $\Delta' = \Omega_2^{-1}(K,K)\circ \Delta$
\begin{align*}
\mu' &: x \otimes x \mapsto \zeta x^2,\\
\Delta' &: x \mapsto 1 \otimes x + x \otimes 1\,\,.
\end{align*}

As an algebra, $L$ is generated by $x$, so the biproduct $r_\cA(\hat{T}_{\zeta,q}) = L \rtimes B$ is generated by $x\cong x \otimes 1$ 
and $g\cong1 \otimes g$. In the biproduct $r_\cA(\hat{T}_{\zeta,q})$,
the relations
\[
g^N=1, x^d = 0 \text{ and } gx = \rho'(g \otimes x) g = qxg
\]
hold. This gives a surjective algebra homomorphism 
$\psi : r_\cA(\hat{T}_{\zeta,q}) \rightarrow \check{T}_{\zeta,q}$; 
since $\check{T}_{\zeta,q}$ and $r_\cA(\hat{T}_{\zeta,q})$  
have the same complex dimension, $\psi$ is an isomorphism.

The map $\psi$ also respects the coalgebra structures, since 
\[
\Delta_{r_\cA(\hat{T}_{\zeta,q})}(x) = 1\cdot x_{[-1]} \otimes x_{[0]} + x \otimes 1 = g^c \otimes x + x \otimes 1 \,\,.
\]
\end{proof}

\subsection{Reflection on simple roots in a Nichols algebra}
\label{sec:ExampleNicholsAlgebra}
We finally discuss the example of Nichols algebras
\cite{HS13}. 
We take for $\cC$  the category of finite-dimensional 
Yetter-Drinfel'd modules over a complex Hopf algebra $h$,
e.g.\ the complex group algebra of a finite group $G$. Let $M\in\cC$
be a finite direct sum  of simple objects $(M_i)_{i\in I}$,
$$M=\bigoplus_{i\in I} M_i\,\, . $$
Thus,  $M$ is a complex braided vector space.
The Nichols algebra $\cB(M)$ of $M$ is defined as a quotient
by the kernels of the quantum symmetrizer maps $Q_n$
$$\cB(M):=\bigoplus_{n\geq 0} M^{\otimes n}/\ker(Q_n).$$
The Nichols algebra $\cB(M)$ is a Hopf algebra  in the braided category $\cC$.
If $M$ is a direct sum of $n$ simple objects in $\cC$, the Nichols 
algebra is said to be of rank $n$.

Each simple subobject $M_i$ of $M$ provides a 
partial dualization datum:
Denote by $M_i^*$ the braided vector space
dual to  $M_i$. Denote by $\cB(M_i)$ the Nichols
algebra for $M_i$. The fact that $M_i$ is a subobject and a quotient
of $M$ implies that $\cB(M_i)$
is a Hopf subalgebra of $\cB(M)$ and that there is
a natural projection $\cB(M)\xrightarrow{\pi_i}\cB(M_i)$ of
Hopf algebras. Similarly,
the evaluation and coevaluation for $M$ induce
a non-degenerate Hopf pairing $\omega_i:\cB(M_i)\otimes\cB(M_i^*)\to \C$
on the Nichols algebras. We thus have for each
$i\in I$ a partial dualization datum
  $$\cA_i:=(\cB(M)\xrightarrow{\pi_i}\cB(M_i),\cB(M_i^*),\omega_i)$$ 
We denote by  $r_i(\cB(M)):=r_{\cA_i}(\cB(M))$ the partial 
dualization of $\cB(M)$ with respect to $\cA_i$.
As usual, we denote by
$K_i$ the coinvariants for the the projection $\pi_i$; 
$K_i$ is a Hopf algebra in the braided category of
$\cB(M_i)$-Yetter-Drinfel'd modules.

We  summarize some results of \cite{AHS10},\cite{HS10a} and 
\cite{HS13}; for simplicity, we assume that the Nichols $\cB(M)$
algebra is finite-dimensional. To make contact with our results, we note that
the $i$-th partial dualization 
$$r_i(\cB(M)):=\Omega(K_i)\rtimes \cB(M_i^*) \,\, , $$
as introduced
in the present paper, coincides by with the $i$-th reflection  
of $\cB(M)$ in the terminology of \cite{AHS10}.

\begin{theorem}
Let $h$ be a complex Hopf algebra. Let $M_i$ be a finite
collection of simple $h$-Yetter-Drinfel'd modules. Consider 
$M:=\bigoplus_{i=1}^n M_i\in\YD{h}$ and assume that the associated
Nichols algebra $H:=\cB(M)$ is finite-dimensional. 
Then the  following assertions hold:
\begin{itemize}
\item By construction, the Nichols algebras 
$\cB(M),r_i(\cB(M))$ have the same dimension as complex vector spaces.
\item
For $i\in I$, denote by $\hat M_i$ the braided subspace
$$  \hat M_i=M_1\oplus \ldots\oplus M_{i-1}\oplus M_{i+1}\oplus\ldots \,\, .$$
of $M$.
Denote by $\ad_{\cB(M_i)}(\hat M_i)$ the braided vector space
obtained as the image of $\hat M_i\subset\cB(M)$ under the adjoint action
of the Hopf subalgebra $\cB(M_i)\subset\cB(M)$. Then,
there is a unique isomorphism \cite[Prop. 8.6]{HS13}  of Hopf
algebras in the braided category 
$\YD{\cB(M_i)}\left(\YD{h}\right)$: 
$$K_i\cong \cB(\ad_{\cB(M_{i})}(\hat M_i))$$
which is the identity on $\ad_{\cB(M_i)}(\hat M_i)$.

\item  Define, with the usual convention for the
sign, \\
$a_{ij}:=-\max\{m\;|\;\ad_{M_i}^m(M_j)\neq0\}$.
Fix $i\in I$ and denote for $j\neq i$
$$V_j:=\ad_{M_i}^{-a_{ij}}\left({M_j}\right)\subset\cB(M)\,\,.$$
The braided vector space 
$$R_i(M)=V_1\oplus\cdots M_i^*\cdots \oplus V_n
\in\YD{h}$$ 
is called the the $i$-th reflection of the braided vector
space $M$.
Then there is a unique isomorphism \cite[Thm.\ 8.9]{HS13}
of Hopf algebras in $\YD{h}$
$$r_i(\cB(M_1\oplus\cdots \oplus M_n))\cong \cB(V_1\oplus\cdots
M_i^*\cdots \oplus V_n)$$
which is the identity on $M$.

\item With the same definition for $a_{ij}$ for $i\neq j$
and $a_{ii}:=2$, the matrix $(a_{ij})_{i,j=1,\ldots n}$ is a
generalized Cartan matrix \cite[Thm. 3.12]{AHS10}. 
Moreover, one has $r_i^2(\cB(M))\cong \cB(M)$,
as a special instance of Corollary \ref{corr:involutive},
and the Cartan matrices coincide, $a_{ij}^M=a_{ij}^{r_i(M)}$. 
In the terminology of \cite[Thm.\ 6.10]{HS10a}, one
obtains  a Cartan scheme. 
\item The maps $r_i$ give rise to a Weyl groupoid which controls the
structure of the Nichols algebra $\cB(M)$. 
For details, we refer to \cite[Sect.\ 3.5]{AHS10}
and \cite[Sect.\ 5]{HS10a}.
\end{itemize}
\end{theorem}

We finally give examples that illustrate the appearance 
of Nichols algebras as Borel algebras in quantum groups.
We end with an example in which a reflected Nichols
algebra is not isomorphic to  the original Nichols
algebra.

The first example serves to fix notation:

\begin{example}\label{ex:5.7}
Let $n>1$ be a natural number and  $q$ be a primitive $n$-th root 
of unity in $\C$. Let $M$ be the one-dimensional 
complex braided vector space with 
basis $x_1$ and braiding matrix $q_{11}=q$. As a quotient of the
tensor algebra, the associated Nichols algebra $\cB(M)$ inherits a 
grading, 
$\cB(M)=\oplus_{k\in\mathbb N}\cB(M)_{(k)}$.
As a graded vector space, it is isomorphic to
$$\cB(M)\cong \C[x_1]/(x_1^n)$$
and thus of complex dimension $n$. The Hilbert series 
is
$$\mathcal{H}(t):=\sum_{k\geq 1} t^k\dim\left(\cB(M)_{(k)}\right)
    =1+t+\cdots t^{n-1}.$$
  \end{example}

The next example exhibits the role of Nichols algebras 
as quantum Borel parts.

\begin{example}
Let $\fg$ be a complex finite-dimensional semisimple Lie 
algebra of rank $n$ with Cartan matrix $(a_{ij})_{i,j=1\ldots n}$. 
Let $(\alpha_i)_{i=1,\ldots,n}$ 
be a set of simple roots for $\fg$ and let
$d_i:=\langle\alpha_i,\alpha_i\rangle/2$.

We construct a braided vector space $M$ with diagonal
braiding as a Yetter-Drinfel'd module over an abelian group:
fix a root $q\neq 1$ of unity; find a diagonal braiding
matrix $q_{ij}$ with
    $$q_{ii}=q^{d_i} \qquad q_{ij}q_{ji}=q_{ii}^{a_{ij}}.$$
The associated Nichols algebra $\cB(M)$ is then the quantum 
Borel part of the Frobenius-Lusztig kernel $u_q(\fg)$.
In this case, all Nichols algebras $r_i(\cB(M))$ obtained by
reflections are isomorphic. The isomorphisms give rise to the 
Lusztig automorphisms $T_{s_i}$ of the algebra $u_q(\fg)$ 
for the simple root $\alpha_i$. These automorphisms 
enter e.g.\ in the
construction of a PBW-basis  for $U(\fg)$.
\end{example}

In the following example  \cite{Heck09}, the two Nichols algebras 
describe two possible Borel
parts of the Lie superalgebra $\fg=\mathfrak{sl}(2|1)$; 
they also appear  in the description \cite{ST13} of 
logarithmic conformal field theories. In this example, 
non-isomorphic Nichols algebras are related by reflections.

\begin{example}
Let $q\neq \pm 1$ be a primitive $n$-th root of unity. Find
two two-dimensional diagonally braided vector spaces
$M,N$, with bases $(x^{(M)}_1,x^{(M)}_2)$ $(x^{(N)}_1,x^{(N)}_2)$
respectively, such that
    \begin{align*}
      q^{(M)}_{11}=q^{(M)}_{22}=-1 \qquad
      & q^{(M)}_{12}q^{(M)}_{21}=q^{-1} \\
      q^{(N)}_{11}=-1\quad q^{(N)}_{22}=q \qquad 
      & q^{(N)}_{12}q^{(N)}_{21}=q^{-1}.
    \end{align*}

We describe a PBW-basis of the Nichols algebras $\cB(M)$ and 
$\cB(N)$ by  isomorphisms of graded vector spaces to
symmetric algebras. To this end, denote for a basis element
$x_i^{(M)}$ of $M$ the corresponding
Nichols subalgebra by $\cB(x_i^{(M)})$, and similarly for
$N$. (We will drop superscripts from now on, wherever they
are evident.) A PBW-basis for the Nichols algebra $\cB(x_i^{(M)})$
has been discussed in Example \ref{ex:5.7}.
Moreover, we need the shorthand
$x_{12}:=x_1x_2-q_{12}x_2x_1$. One can show that the 
multiplication in the Nichols algebras leads to isomorphisms
of graded vector spaces: 
    \begin{align*}
      \cB(M)
      &\stackrel\sim\leftarrow \cB(x_1)\otimes \cB(x_2)\otimes 
      \cB(x_1x_2-q_{12}x_2x_1)\\
      & \cong \C[x_1]/(x_1^2)\otimes \C[x_2]/(x_2^2) \otimes
        \C[x_{12}]/(x_{12}^n),\\
      \cB(N)
      &\stackrel\sim\leftarrow \cB(x_1)\otimes \cB(x_2)\otimes 
      \cB(x_1x_2-q_{12}x_2x_1)\\
      & \cong\C[x_1]/(x_1^2)\otimes \C[x_2]/(x_2^n) \otimes
        \C[x_{12}]/(x_{12}^2).
    \end{align*}
Both Nichols algebras $\cB(M)$ and $\cB(N)$ are of dimension
$4n$ and have a Cartan matrix of type $A_2$.
Their Hilbert series can be read off from the
PBW-basis:     
\begin{align*}
     \mathcal{H}_{\cB(M)}(t)
     &=(1+t)(1+t)(1+t^2+t^4\cdots t^{2(n-1)}),\\
     \mathcal{H}_{\cB(N)}(t)
     &=(1+t)(1+t+t^2+\cdots t^{n-1})(1+t^2).\\
    \end{align*}
The two Hilbert series are different; thus the two Nichols algebras 
$\cB(M)$ and $\cB(N)$ are {\em not} isomorphic. The Nichols
algebras are, however, related by
partial dualizations:
$$r_1(\cB(M))=\cB(N)\qquad r_2(\cB(M))\cong \cB(N)$$
$$r_1(\cB(N))=\cB(M)\qquad r_2(\cB(N))=\cB(N)\,\,$$
%CS ---------^^ equal or iso
where $r_i$ is the partial dualization with respect to the
subalgebra $\cB(\C x_i)$.
For the isomorphism indicated by $\cong$, the generators
$x_1$ and $x_2$ have to be interchanged.
  
  \end{example}

\noindent{\sc Acknowledgments}: 
We are grateful to Yorck Sommerh\"auser for many helpful
discussions.
The authors are partially supported by the DFG Priority 
Programme SPP 1388 ``Representation Theory''.


\begin{thebibliography}{12345}
  \bibitem[AF00]{AF00} 
    J.N. Alonso Alvarez, J.M. Fern{\'a}ndez Vilaboa:
    Cleft extensions in braided categories, 
    Comm. Algebra 28 (2000) 3185-3196
  \bibitem[AHS10]{AHS10}
    N. Andruskiewitsch, I. Heckenberger, H.-J. Schneider:
    The Nichols algebra of a semisimple {Y}etter-{D}rinfeld module,
    Amer. J. Math. 132 (2010) 1493-1547
  \bibitem[AS02]{AS02} 
    N. Andruskiewitsch, H.-J. Schneider:
    Pointed {H}opf algebras, in: ``New Directions in Hopf 
    Algebras''
MSRI Publications 43  (2002) 1-68.
    Cambridge University Press, 2002.
  \bibitem[B95]{Besp95}
    Yu. N. Bespalov:
    Crossed modules, quantum braided groups and ribbon structures,
    Teoret. Mat. Fiz. 103 (1995)  368-387
  \bibitem[BV12]{BV12}
    A. Brugui{\`e}res, A. Vir{\'e}lizier:
    The doubles of a braided Hopf algebra,
in: ``Hopf algebras and tensor categories'',
Proceedings of the International Conference held at the 
University of Almer\'\i{}a, 2011. N.\ Andruskiewitsch,
J. Cuadra and B.  Torrecillas (eds).
   Contemp. Math. 585 (2013) 175-197
  \bibitem[ENO11]{ENO11}
    P. I. Etingof, D. Nikshych, V. Ostrik:
    Weakly group-theoretical and solvable fusion categories,
    Adv. {M}ath. 226 (2011) 176-205
  \bibitem[H09]{Heck09}
    I. Heckenberger:
    Classification of arithmetic root systems,
    Adv.  {M}ath. 220 (2009) 59-124
  \bibitem[HS10]{HS10a}
    I. Heckenberger, H.-J. Schneider:
    Root systems and {W}eyl groupoids for {N}ichols algebras,
    Proc. Lond. Math. Soc. 101 (2010) 623-654
  \bibitem[HS13]{HS13}
    I. Heckenberger, H.-J. Schneider:
    {Y}etter-{D}rinfeld modules over bosonizations of dually paired {H}opf algebras,
    Adv. {M}ath. 244 (2013) 354-394
  \bibitem[K95]{Kass95}
    C. Kassel:
    Quantum groups,
    Graduate Texts in Mathematics 155, Springer, 1995.
  \bibitem[Mo93]{Mont93}
    S. Montgomory:
    Hopf algebras and their actions on rings,
    AMS and CBMS, 1993.
  \bibitem[Mu03]{Mue03}
    M. M{\"u}ger:
{}From subfactors to categories and topology {I}. {F}robenius algebras in and {M}orita equivalence of tensor categories,
    J. Pure Appl. Algebra 180 (2003) 81-157
  \bibitem[S01]{Schau01}
    P. Schauenburg:
    The monoidal center construction and bimodules,
    J. Pure Appl. Algebra 158 (2001) 325-346
  \bibitem[ST13]{ST13}
    A.M. Semikhatov, I. Yu. Tipunin:
    Logarithmic $\hat{\mathfrak{sl}}_2$ {CFT} models from {N}ichols algebras,
    arXiv:1301.2235, Preprint 2013.
\end{thebibliography}
\end{document}